\documentclass[graybox]{svmult}

\usepackage{type1cm}        
\usepackage{makeidx}         
\usepackage{graphicx}        
\usepackage{multicol}        
\usepackage[bottom]{footmisc}

\usepackage{newtxtext}       %
\usepackage[varvw]{newtxmath}       

\makeindex             

\usepackage[makeroom]{cancel}
\usepackage{bbm, dsfont}

\usepackage{amsmath}

\begin{document}

\title*{Data representation with optimal transport}
\author{Roc\'io D\'iaz Mart\'in\orcidID{0000-0002-3732-6296} and\\ Ivan Vladimir Medri\orcidID{0000-0003-2419-2193} and\\  Gustavo Kunde Rohde\orcidID{0000-0003-1703-9035}}
\institute{Vanderbilt University \at Dept. of Mathematics, 1326 Stevenson Center, Station B 407807, Nashville, TN 37240. \email{rocio.p.diaz.martin@vanderbilt.edu}
\and University of Virginia \at Dept. of Biomedical Engineering, Thornton Hall, 351 McCormick Road,
Charlottesville, VA 22904. \email{pzr7pr@virginia.edu}
\and University of Virginia \at Dept. of Biomedical Engineering and Electrical and Computer Engineering, Thornton Hall, 351 McCormick Road,
Charlottesville, VA 22904. \email{gustavo@virginia.edu}}
%
%
\maketitle

\abstract*{Optimal transport has been used to define bijective nonlinear transforms and different transport-related metrics for discriminating data and signals. Here we briefly describe the advances in this topic with the main applications and properties in each case.}

\section{Introduction}
\label{sec:intro}

Finding mathematical formulas for representing data (such as signals, images, vectors, measures, etc.) has long interested mathematicians, engineers, and scientists. The underlying premise is that certain mathematical representations can facilitate solutions to various problems. To achieve this, operators that map the original space of the objects of interest into a space of representations are often defined and referred to as \textit{transforms}.

For instance, in the field of Signal Processing Analysis\footnote{In signal processing, the term \textit{analysis} typically replaces the notion of \textit{representation}.}, there are many examples commonly used: 
(1) The Fourier transform represents signals in terms of frequency components, allowing for the practical solution of problems related to shift-invariant linear operators, including the heat (diffusion) equation, convolution problems, and numerous others \cite{stein2011fourier}.
(2) The Laplace transform represents derivatives in a way that allows the conversion of initial value problems into algebraic problems \cite{siebert1986circuits}.
(3) In medical imaging, if a function describes an unknown density in a patient, we can also represent it by the collection of all the X-ray projections taken at different angles around the patient, and use the Radon transform for reconstruction purposes \cite{helgason2011integral}.
(4) The wavelet transform analyzes signals in time by introducing two new variables, namely scale, and resolution \cite{mallat1999wavelet}. The listed mathematical transformations (Fourier, Laplace, Radon, and Wavelet transform) are linear operators, and thus often fail to deal with the non-linearities present in modern data science applications. There are some exceptions to this shortcoming. For example, the scattering transform is non-linear and has been
successfully applied to machine learning applications \cite{mallat2012group}.



In this chapter, we will describe a family of non-linear transforms rooted in the theory of optimal transport (OT) (see also \cite{kolouri2016transport}). In addition to obtaining new representations through these transforms, they will allow us to create new metrics or distances to compare our original signals or data (see, for e.g., \cite{rubaiyat2020parametric,shifat2021radon, rubaiyat2024end, basu2014detecting,kundu2020enabling}).   

The theory of optimal transport addresses matching problems between different configurations of mass by posing a minimization problem \cite{sant2015,Villani2003Topics,Villani2009Optimal}. 
The total cost that results from solving the optimization problem introduces new tools for comparing signals: the so-called Wasserstein distances. Since these distances arise from alignment problems, they are often more `natural' for comparing probability densities or measures than typical Euclidean distances (i.e., 
$L^2$-distances). However, Wasserstein distances are difficult and expensive to compute.

The transforms we will describe in this chapter serve as a trade-off between transport theory and classical 
$L^2$-theory. We will interpret our data or signals as measures and embed the space of measures into an 
$L^2$-space. This embedding will be a one-to-one transform that relates to optimal transport. These \textit{transport transforms}, which we will usually denote by $\widehat{\cdot}$, will enable us to define metrics in the original space of measures that retain some natural properties of Wasserstein distances while also being expressible as a simple Euclidean norm, thus allowing the computational framework of Hilbert spaces. See Figures \ref{fig: transform diagram} and \ref{fig: tangent}.

\begin{figure}
    \centering
    \includegraphics[width=0.9\linewidth]{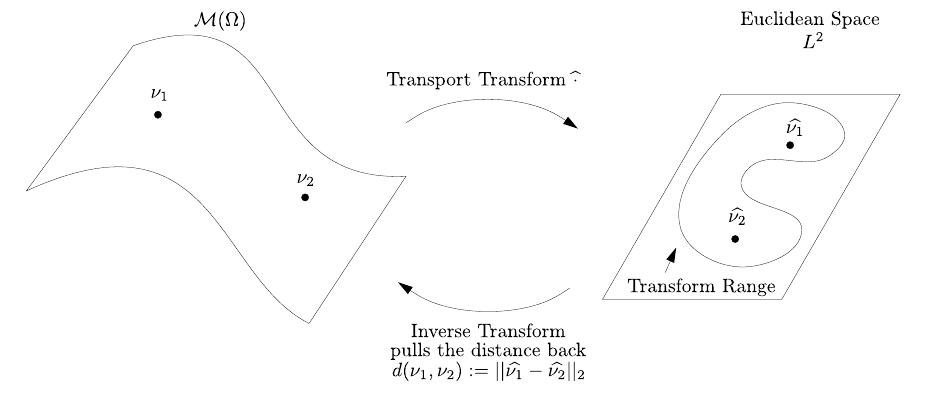}
    \caption{The space of measures  $\mathcal{M}(\Omega)$ defined on $\Omega\subseteq\mathbb{R}^d$ is embedded into a Euclidean space through a \textit{transport transform} $\widehat{\cdot}$. By pulling the $L^2$-norm back we define a distance $d(\cdot,\cdot)$ in the space of measures.}
    \label{fig: transform diagram}
\end{figure}

\noindent\textbf{Motivating tasks.}
Let us assume that in a physical system, a density function (that might represent mass at every point in space) suffers deformations given by translations and dilations without modifying the total mass. Thus, it would take the form $f_{(\omega,\tau)} (x) = \omega f(\omega x-\tau)$, where $f$ is a \textit{template} configuration and $(\omega,\tau)$ are deformation parameters.
Suppose that we experimentally measure this quantity $f_{(\omega, \tau)}$ obtaining an estimated density $h$.
Some example questions that transport transforms can help to solve are:
\begin{enumerate}
    \item (Classification.) If the initial configuration $f$ is unknown but lies within a finite number of {template} options $\{f_k\}_{k=1}^K$, how can we identify the class $k$ of our system?
    \item (Estimation.) If the initial template is known, how can we find the correct deformation parameters $(\omega,\tau)$? 
    \item (Reconstruction or generation.) If the template $f$ and the deformation parameters $(\omega,\tau)$ are known, is it possible to reconstruct the intermediate steps that changed the template $f$ to $f_{(\omega,\tau)}$ in a `natural' way? 
\end{enumerate}


The first task is a {classification problem} where each class is a collection of functions of the form $\mathbb{S}_k := \{ x\mapsto \omega f_k(\omega x-\tau): \, \omega\in\mathbb R_{>0}, \tau\in\mathbb R\}$ with $f_k$, $k=1,\dots,K$, fixed templates. Those classes are not linearly separable as can be seen in the left panel of Figure \ref{fig: transform diagram classes}. Nevertheless, if we apply certain transport transform, the transformed classes take the form  $\widehat{\mathbb{S}_k} = \{\xi\mapsto a\widehat{f_k}(\xi) + b: \, a\in\mathbb{R}_{>0}, b\in\mathbb R\}$ (as will be explained in Section \ref{sec:embeddings_transforms}) creating half-planes, which are linearly separable in transform space as shown in the right panel of Figure \ref{fig: transform diagram classes}. 

\begin{figure}[h!]
    \centering
    \includegraphics[width=\linewidth]{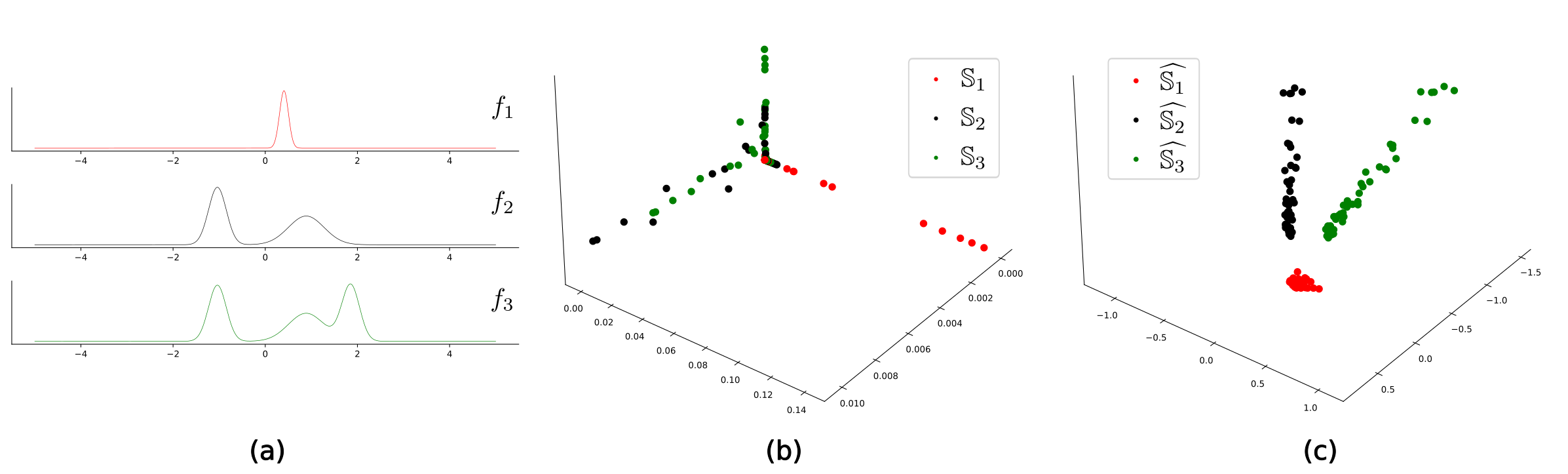}
    \caption{(a) Represents three template functions $f_1,f_2,f_3$. (b) Represents the translations and dilations of the templates, $\mathbb{S}_1, \mathbb{S}_2, \mathbb{S}_3$, in 3D space using three coordinate values. (c) Represents the transforms of translations and dilations of the templates, $\widehat{\mathbb{S}_1}, \widehat{\mathbb{S}_2}, \widehat{\mathbb{S}_3}$, in 3D space.}
    \label{fig: transform diagram classes}
\end{figure}


For the second point, when the template $f$ is known, we seek for $(\omega,\tau)$ so that $f_{(\omega,\tau)}$ best matches the measured signal $h$.  A naive attempt to solve this problem would be to minimize the functional $J_1(\omega,\tau) = \|f_{(\omega,\tau)} - h \|_2^2$. 
Nevertheless, this is a non-convex problem (with respect to $(\omega,\tau)$) and global minima might be hard to find. With the use of transport transforms, a different functional $J_2(a,b) = \|a\widehat{f} + b -\widehat{h}\|_2^2$ with $a = 1/\omega, b = \tau/\omega$ can be utilized. Minimizing $J_2$ is a convex problem, a global minimum is easily found, and the parameters $(\omega,\tau)$ can be obtained afterward from $({a},{b})$. Using the injectivity and properties of the transform, it can be proven that the parameters found with this method are exactly the ones that make $h = f_{(\omega,\tau)}$.





The third task is an interpolation problem as we want to generate intermediate signals between the template $f$ and the target $f_{(\omega,\tau)}$. 
Since the transport transforms are embeddings in $L^2$ spaces, the interpolation can be obtained as a simple convex combination $(1-t)\widehat{f}+t\widehat{f_{(\omega,\tau)}}$, for $0\leq t\leq 1$, and then applying the inverse transform. Figure \ref{fig: transform diagram interpol} compares such transition with the convex combination  $(1-t)f+t{f_{(\omega,\tau)}}$ taken directly in the original signal space.

\begin{figure}[h!]
    \centering
    \includegraphics[width=\linewidth]{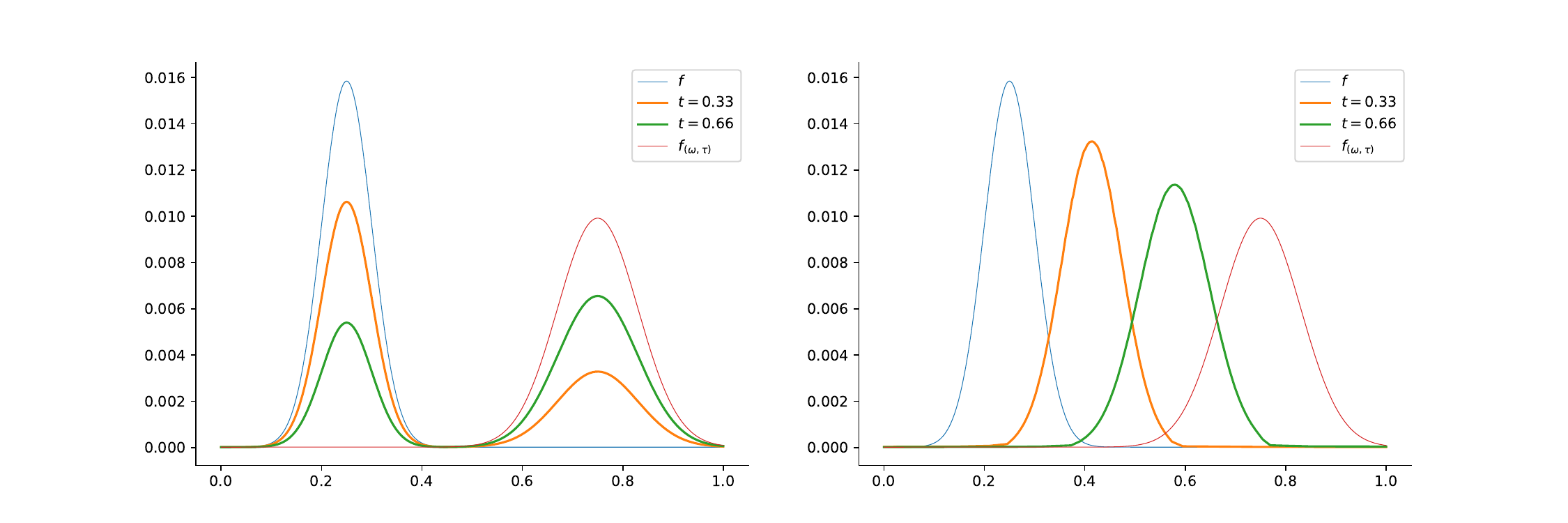}
    \caption{(Left) Interpolation in original space. A Gaussian function $f$ that moves to another Gaussian $f_{(\omega,\tau)}$ has intermediate states given by a combination of two Gaussians. (Right) A Gaussian function $f$ moves to another Gaussian $f_{(\omega,\tau)}$ using interpolation in transform space. Intermediate densities behave in accordance with the initial and final configurations ($f$ and $f_{(\omega,\tau)}$).}
    \label{fig: transform diagram interpol}
\end{figure}



The underlying concept among the above-mentioned problems is that a template function is being altered through mass-preserving transformations.
Although we described our examples using dilations and translations, more complicated deformations can be considered. Usually,  they are associated with changes in the domain of the functions that are hard to express in a simple form for classification, estimation, or reconstruction. Transport theory conveniently captures these deformations, and transport transforms will allow us to express them in a simplified formulation that can utilize the computational power of Euclidean space methods. The advantages illustrated above will be detailed in this chapter. These are not exhaustive; other benefits include the possibility of performing PCA in transform-space and optimizing costs using gradient descent, among others.\\

\noindent \textbf{Signals and images, point clouds, sampled data, density functions, and measures.}
The subject matter of this chapter is related to the mathematical modeling of signals, images, and data in general. We will adopt \textit{measure theory} as a unifying framework. 
Measures $\mu$ generalize considering signals and images as functions $f_\mu$.
In many cases, the physical meaning of $f_\mu$ is related to the \textit{density} of a certain quantity (mass, energy, attenuation, pressure, etc.\footnote{For example, in fluorescence microscopy, the amount of light (i.e., the number of photons) hitting a detector at location $x$ is directly proportional to the amount of fluorescently tagged protein in the optical path at that location.
Similarly, in an image obtained by an optical camera, $f_\mu(x)$ refers to the optical flux per unit area (radiant exitance) \cite{barrett2013foundations}. 
In 3D magnetic resonance imaging (MRI), the intensity of the proton density image $f_\mu(x)$ is directly proportional to the proton density at the location $x$, which means that the image intensity $f_\mu(x)$ reflects the concentration of protons at each point in the scanned volume.
Similar interpretations can be made about numerous other sensing and imaging modalities, including nuclear medicine (density of a radioactive tracer), X-ray computed tomography (a 2D or 3D distribution of attenuation coefficients), and sound intensity (pressure per unit area per time), among other examples.}). 
The measure $\mu$ with density $f_\mu$ computes the total amount of that quantity within a region $A \subseteq \mathbb{R}^d$ 
via
\begin{equation} \label{eq: measure with density}
  \mu(A) = \int_A f_\mu(x)\, dx.  
\end{equation}
For digital data (e.g., discrete images\footnote{A two-dimensional grayscale image is usually regarded as a function mapping a spatial domain  $\Omega \subseteq \mathbb{R}^2$ to $\mathbb{R}$. Typically, when using pixels, $\Omega$ is a finite discrete set.}, point clouds\footnote{Point clouds are collections of points representing object surfaces often produced 3D. Each point $x$ encodes a position in Cartesian coordinates and can be viewed as a delta measure $\delta_x$. 
}, or when considering sampled data) we may have intensity measurements $f_\mu(x_k)$ only at locations $x_k \in \mathbb{R}^d, k=1, \dots, K$. Thus, we can view our data as a discrete measure $\mu = \sum_{k=1}^K f_\mu(x_k) \delta_{x_k} $ (that is, $\mu(A) = \sum_{k=1}^K f_\mu(x_k) \delta_{x_k}(A)$, where $\delta_{x_k}(A)=1$ if $x_k\in A$ and 0 otherwise). 

Measures are the natural objects when studying optimal transport theory. Thus, as our transforms will be motivated by mass-transportation problems, using measure theory will simplify the notation and formula as will be seen later in this chapter. 
\\

\noindent\textbf{Chapter Organization.} In Section \ref{sec: ot}, we provide a brief introduction to Optimal Transport. 
Section \ref{sec:embeddings_transforms} starts by synthesizing the general framework of transport transforms, and later summarizes its variants related to closed-form solutions for Optimal Transport in 1D. Applications of these tools are given in Section \ref{sec: app}, where we revisit the motivating tasks stated in this introduction.

\section{Formulation of the Optimal Transport problem}\label{sec: ot}


\noindent\textbf{Monge's formulation.}
Let $\mathcal{P}(\Omega)$ be the set of probability measures defined in a subset $\Omega\subseteq \mathbb{R}^d$, and consider $\mu$, $\nu \in \mathcal{P}(\Omega)$. The Monge Optimal Transport (OT) problem between $\mu$ and $\nu$ is to find the most economical procedure to move all the mass distributed according to $\mu$ onto the target $\nu$ by using a map $T$. Mathematically, it is stated as the minimization problem 
    \begin{equation*}
        OT(\mu,\nu) = \inf_{T \in \mathrm{MP}(\mu,\nu)} \int_{\Omega} c(x,T(x)) d\mu(x),
    \end{equation*}
where $c:\Omega\times\Omega \to \mathbb{R}_{\geq 0}$ represents the cost of transporting one unit of mass from position $x$ to $y= T(x)$, and $\mathrm{MP}(\mu,\nu)$ is the set of all \textit{mass-preserving} functions, that is,  measurable functions $T:\Omega\to\Omega$ that push measure $\mu$ into $\nu$ in the sense that 
    \begin{equation*}\label{eq: pushforward}
        \nu(B) = \mu(T^{-1}(B)) =\mu(\{x: \, T(x)\in B\}) \quad \text{ for every $\nu$-measurable set $B$.}
    \end{equation*}
If $T$ satisfies this relation, it is called \textit{transport map}, and  we say that $\nu$ is the pushforward of $\mu$ by $T$, denoted as $\nu = T_\#\mu$
\footnote{$\nu = T_\#\mu$ can also be defined through the change of variables formula
    \begin{equation*}
       \int \psi(y)\, d\nu(y)= \int \psi(y) \, d(T_\#\mu)(y) = \int \psi(T(x)) \, d\mu(x) 
    \end{equation*}
    for every test function $\psi$. Also, 
    if $X$ and $Y$ are random variables with distributions $\mu$ and $\nu$, respectively, we have that $Y=T(X)$. 
}. 
They encode that all mass at point $x$ should be moved to $y=T(x)$. 
In the special case where $\mu$ and $\nu$ have density functions $f_\mu$ and $f_\nu$ (i.e., $\mu$ is defined according to \eqref{eq: measure with density}, and analogously for $\nu$),  the above relation can also be written as 
\begin{equation*}
    \int_{T^{-1}(B)} d\mu = \int_{T^{-1}(B)} f_\mu(x)dx = \int_{B} f_\nu(x)dx = \int_{B}d\nu.
\end{equation*}
In addition, if $T$ is smooth and one-to-one,  using the change of variable formula and denoting by $DT$ the Jacobian of $T$, we obtain  $\nu=T_\#\mu$ if and only if 
\begin{equation*}\label{eq: pushforward for densities}
   |\mathrm{det}(DT(x))| \,  f_\nu(T(x)) = f_\mu(x).
\end{equation*}
A closed-form solution for Monge's problem can be obtained when working with one-dimensional measures. Given $\nu \in \mathcal{P}(\mathbb{R})$ we define its cumulative distribution function (CDF) as $F_\nu(x) := \nu([-\infty, x))$ together with its \textit{generalized inverse}\footnote{To be more precise, when considering $\nu\in\mathcal{P}(I)$ where  $I=[a,b]$ is a real interval, we define the generalized inverse of $F_\nu:[a,b]\to[0,1]$ as the function $F_\nu^\dagger:[0,1] \to [a,b]$ given by
            $F_\nu^\dagger(y) := \inf \{x\in [a,b]: \, F_\nu(x) > y \}$, 
    where we impose $\inf \emptyset = b$. Notice that it is similar to the \textit{quantile} function the distribution $\nu$.}  
\begin{equation*}
    F_\nu^{\dagger}(y) := \inf\{x\in \mathbb{R}: \, F_\nu(x) > y \}.
\end{equation*}
Note that when $F_\nu$ is invertible, the generalized inverse coincides with the inverse. Otherwise, for non-decreasing functions, it is just the function whose graph is a reflection through $y = x$ of the graph of $F_\nu$ considering that flat regions are reflected into a discontinuous jump and vice-versa. With these definitions, in the case when $\mu$ absolutely continuous with respect to the Lebesgue measure ($\mu\ll\mathcal{L}$) an optimal Monge map from $\mu$ to $\nu$ can be explicitly obtained as 
\begin{equation}\label{eq: gen inverse comp with cdf}
    T(x) = F_\nu^{\dagger}(F_\mu(x)).
\end{equation}   
In terms of densities, that is, if $\mu$ and $\nu$ have probability density functions (pdf) $f_\mu$ and $f_\nu$, respectively, $T$ in \eqref{eq: gen inverse comp with cdf} satisfies the following integral equation (see Figure \ref{fig: Monge OT}):
\begin{equation}\label{eq: preserve areas}
   \underbrace{\int_{-\infty}^{x} f_\mu(x)\, dx}_{F_\mu(x)}=\int_{-\infty}^{x}f_\nu(T(x))T^\prime(x)\, dx=\underbrace{\int_{-\infty}^{T(x)}f_\nu(y)
   \, dy}_{
   F_{\nu}(T(x))}
\end{equation}

\begin{figure}[h!]    
  \centering
    \includegraphics[width=0.9\linewidth]{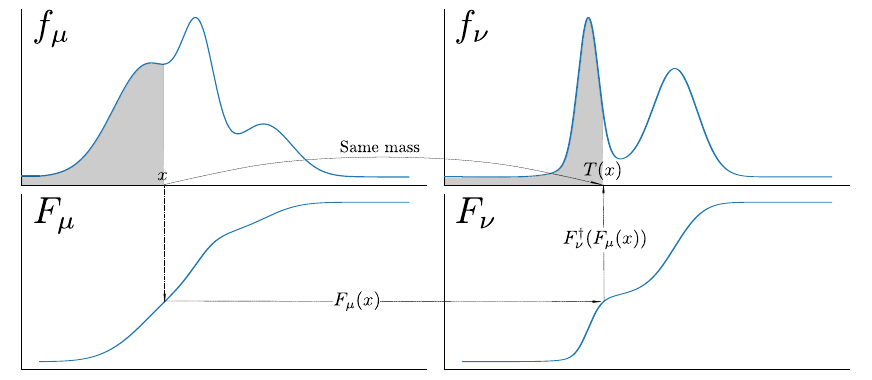}

     \caption{Illustration of mass transportation from reference $\mu$ to target $\nu$ through a MP map $T$ in 1D. The measures $\mu$ and $\nu$ have pdfs $f_\mu$ and $f_\nu$, and CDFs $F_\mu$ and $F_\nu$, respectively. The gray shadowed areas coincide. They represent the accumulated mass until point $x$ (on the left) and until point $T(x)$ (on the right). To achieve this conservation of mass, the increasing function $T$ is obtained as \eqref{eq: preserve areas}. Such formulation \eqref{eq: preserve areas} is visualized in this figure: $x\mapsto F_\mu(x)\mapsto F_\nu^\dagger(F_\mu(x))=T(x)$.}
    \label{fig: Monge OT}
\end{figure}

\noindent\textbf{Kantorovich's formulation.} For higher dimensions, Monge's optimal transport problem has no closed-form solution. In addition, it is sometimes ill-posed when working with discrete measures. For example, there is no function $T$ that pushes a single delta measure $\mu = \delta_0$ onto $\nu = 0.5 (\delta_{-1} + \delta_{1})$ since it would require to split the mass at $x=0$ into two separate targets $y = \pm 1$. Besides, the constraint  $T:X\to Y$ such that $\nu = T_\# \mu$ is non-linear. This makes the minimization procedure rather convoluted. 
An alternative formulation of the optimal transport problem that solves most of these issues was proposed by Kantorovich. Specifically, Kantorovich's optimal transport problem is stated as the minimization problem   
\begin{equation}\label{eq: kantorovich}
    OT(\mu,\nu) = \inf_{\pi \in \Pi(\mu,\nu)} \int_{\Omega^2} c(x,y) d\pi(x,y),
\end{equation}
where $\Pi(\mu,\nu)$ is the set of all joint probability measures in $\Omega^2$ with marginals $\mu$ and $\nu$, that is, $\pi(A\times \Omega) = \mu(A)$, $\pi(\Omega\times B) = \nu(B)$ for all measurable sets  $A, B \subseteq \Omega$. A measure $\pi \in \Pi(\mu,\nu)$ is called a \textit{transportation plan}, and the quantity $\pi(A\times B)$ represents how much mass from $A$ should be moved to region $B$. 
Now, when working with the example mentioned before\footnote{In general, for the discrete case,    $\mu=\sum_{i=1}^n\mu(x_i)\delta_{x_i}$,  $\nu=\sum_{j=1}^m\nu(y_j)\delta_{y_j}$ the problem \eqref{eq: kantorovich} reads as    ${\min\sum_{i=1}^n\sum_{j=1}^mc(x_i,y_j) \, \pi(x_i,y_j)}$ subject to   $\sum_{i=1}^n\pi(x_i,y_j)=\nu(y_j)$,  $\sum_{j=1}^m\pi(x_i,y_j)=\mu(x_i)$.
Understanding the discrete plans as matrices $\pi\in \mathbb{R}^{n\times m}$, this becomes a linear programming.}, moving $\mu = \delta_0$ onto $\nu = 0.5 (\delta_{-1} + \delta_{1})$ can be done by using a plan of the form $\pi = 0.5\delta_{0,1}+ 0.5\delta_{0,-1}$. 


The set of all transportation plans is non-empty as $\pi=\mu\times\nu\in \Pi(\mu,\nu)$. A minimizer of \eqref{eq: kantorovich} is called an optimal transport plan and its existence is guaranteed when the cost function $c$ is continuous.
Nevertheless, optimal plans are not necessarily unique.
For our exposition, we will work with the Euclidean cost $c(x,y) = \|x-y\|^2$ which is strictly convex.  To guarantee \eqref{eq: kantorovich} to be finite, we will restrict the space of measures to those with finite second moment
\begin{equation*}
    \mathcal{P}_2(\Omega) := \left\{\nu \in \mathcal{P}(\Omega): \, \int_\Omega \|x\|^2 d\nu <\infty \right\}.
\end{equation*}
In this case,
a stronger result can be obtained that guarantees the existence, uniqueness, and finiteness of the solution of Kantorovich and Monge's problems simultaneously.
A space of functions that will be useful in the next theorem and the rest of the chapter is the space of square-integrable functions that are gradients of convex functions
\begin{eqnarray}\nonumber
    L_{G}^2(\mu):= \{T:\mathbb{R}^d\to \mathbb{R}^d : \, T=\nabla \varphi, \, \varphi \text{ convex and } \int \|T(x)\|^2d\mu(x) <\infty\}. \nonumber
\end{eqnarray}

\begin{theorem}[Brenier's theorem] \label{thm: Brenier}
    Let $\mu,\nu\in\mathcal{P}_2(\mathbb{R}^d)$ and $c(x,y) = \|x-y\|^2$. Suppose that $\mu$ is absolutely continuous with respect to the Lebesgue measure in $\mathbb{R}^d$ ($\mu\ll\mathcal{L}$). Then, there exists a unique solution to Kantorovich's problem. It is of the form  $\pi = (\mathrm{id}, T)_\# \mu$, with $T\in L_{G}^2(\mu)$. Moreover, $T$ is also the unique solution to Monge's problem, and $T_\#\mu$, and we will denote it by $T=T_{\mu\to\nu}$. Conversely, let $\mu \in \mathcal{P}_2(\mathbb{R}^d)$ and $T \in L_{G}^2(\mu)$, then $T$ is optimal for the Monge Problem between the measures $\mu$.     
\end{theorem}

\begin{remark}
    Let $\mu,\nu$ as in Brenier's Theorem. If $T_1=\nabla \varphi_1$ and $T_2=\nabla \varphi_2$ for $\varphi_1,\varphi_2$ convex functions, and $(T_1)_\#\mu=\nu=(T_2)_\#\mu$, then $T_1=T_2$.
\end{remark}



\noindent\textbf{Wasserstein distance.} Intuitively speaking, moving a distribution of mass $\mu$ to $\nu$ in an optimal way should cost less than moving $\mu$ to a third measure and then to $\nu$, which resembles a triangle inequality. In fact, in $\mathcal{P}_2(\Omega)$ the square root of the transport cost $OT(\cdot,\cdot)$ defines the so called Wasserstein metric
    \begin{equation*}
        W_2(\mu,\nu) := OT(\mu,\nu)^{1/2}.
    \end{equation*}

\noindent\textbf{Dynamic formulation.} 
The Wasserstein distance renders $\mathcal{P}_2(\Omega)$ a metric space. This structure is very rich geometrically. For simplicity and from now on, let us assume that $\Omega\subset\mathbb{R}^d$ is convex and compact. $(\mathcal{P}(\Omega), W_2)$ is also a geodesic space: given two measures  $\mu,\nu\in\mathcal{P}(\Omega)$, there exists a curve of measures $\rho_t\in\mathcal{P}(\Omega)$, for $0\leq t\leq 1$, that is the shortest path in terms of $W_2$ connecting them ($\rho_0=\mu$, $\rho_1=\nu$) and, moreover, the length of $\rho_t$ is exactly $W_2(\mu,\nu)$.
 Furthermore, as an infinite-dimensional manifold, $(\mathcal{P}(\Omega), W_2)$ possesses a Riemannian structure (its tangent spaces are equipped with an inner product).
A precise definition of these structures 
requires the dynamic approach \cite{benamou2000computational} of the Optimal Transport problem.

Kantorovich and Monge's approaches provide static formulations of the transport problem. Roughly speaking, transport plans and maps give a rule that assigns to each initial mass a final point where it must be moved, but they do not say how the system should evolve from the initial to final configurations. 

In the dynamic formulation, we consider a curve of measures $\rho \in \mathcal{P}([0,1]\times\Omega)$ that for each time $t\in [0,1]$ gives a measure $\rho_t:=\rho(t,\cdot)\in \mathcal{P}(\Omega)$. To transport mass from $\mu$ to $\nu$ we will require the curve to have initial and final conditions $\rho_0 = \mu$ $\rho_1 = \nu$, to be sufficiently smooth, and to satisfy a conservation of mass law: the curve $\rho_t$ together with a velocity field $v_t:= v(t,\cdot)$ must satisfy the \textit{continuity equation} with boundary conditions\footnote{To impose $\rho$ to preserve the mass, we can state that for every  $B\subset \Omega$, we should have $\frac{d}{d t} \int_B \rho_t d V=-\int_{\partial B} \rho_t v_t \cdot d S$, where $dV$ and $dS$ denote volume and surface integrals, respectively. By using the Divergence Theorem we obtain
$\frac{d}{d t} \int_B \rho_t d V=-\int_B \nabla \cdot \rho_t v_t d V$ from where we can derive the continuity equation \eqref{eq: continuity eq}.} 
\begin{equation}\label{eq: continuity eq}
    \partial_t\rho + \nabla \cdot (\rho v) = 0, \qquad \rho_0 = \mu, \quad \rho_1 = \nu.    
\end{equation}
Then, the Optimal Transport problem can be stated as minimizing the following \textit{kinetic energy}
\begin{equation}\label{eq: OT dynamic}
  OT(\mu,\nu)=\inf\left\{\int_{0}^1\int_{\mathbb{R}^d} \|v_t(x)\|^2 d\rho_t(x)dt: \, (\rho,v) \text{ satisfies  \eqref{eq: continuity eq}}\right\}  
\end{equation}

Since the space $\mathcal{P}(\Omega)$  has the metric $W_2$, we can speak about the length of curves. Formally, for a curve $\rho \in \mathcal{P}([0,1]\times\Omega)$ we would define its length as the limit 
\begin{align*}
    \text{length}(\rho) &\approx 
    \lim_{\Delta t \to 0} \sum \frac{W_2(\rho_{t+\Delta t},\rho_t)}{\Delta t} \Delta t\approx \int_0^1 \lim_{\Delta t \to 0} \frac{W_2(\rho_{t+\Delta t},\rho_t)}{\Delta t} \ dt. \label{eq: formal integral length}
\end{align*}
In particular, for a curve $\rho$ solving  \eqref{eq: continuity eq} with velocity filed $v$, 
the length of the curve  can be computed as 
\begin{equation*}\label{eq: interal length}
     \text{length}(\rho) = \int_{0}^1\int_{\mathbb{R}^d} \|v_t(x)\|^2 d\rho_t(x)dt.
\end{equation*}


Under the assumption of the existence of an optimal Monge map $T_{\mu\to\nu}$, an optimal solution $(\rho,v)$ for \eqref{eq: OT dynamic} can be given explicitly. If a particle starts at position $x$ and finishes at position $T_{\mu\to\nu}(x)$, then for $0<t<1$ it will
be at the point $T_t(x)=(1-t)x +tT_{\mu\to \nu}(x).$
Varying both the time $t\in[0,1]$ and $x\in\Omega$, the mapping $(t,x)\mapsto T_t(x)$ 
can be interpreted as a flow whose time velocity 
 is 
\begin{equation}
    v_t(x) = T_{\mu\to\nu}(x_0)-x_0,  \qquad  \text{ for } x=T_t(x_0).\label{eq: ot v}  
\end{equation}
Then, to obtain the curve of probability measures $\rho_t$, one can evolve $\mu$ through the flow $T_t$ by using the formula 
\begin{equation}\label{eq: ot rho} 
     \rho_t=(T_t)_\#\mu=((1-t)\mathrm{id}+tT_{\mu\to\nu})_\#\mu, \qquad 0\leq t\leq 1.  
\end{equation}
It holds that this pair $(\rho,v)$ 
satisfies the continuity equation \eqref{eq: continuity eq} and solves \eqref{eq: OT dynamic}. The vector field $v_t$ can be viewed as 
the tangent vector to the evolution curve at time $t$. With this interpretation, the field $v_0(x) = T_{\mu\to\nu}(x)-x$ (which encodes the optimal displacement) is the initial tangent vector to the optimal curve transporting $\mu$ to $\nu$. In symbols, $v_0\in\mathrm{Tan}_\mu$. By determining $v_0$, from \eqref{eq: ot v} we know $v_t$ for every $t$. Also, it holds that $\text{length}(\rho)=\int_{\mathbb R^d} \|v_0(x)\|^2d\rho_0(x)=\int_{\mathbb R^d} \|T_{\mu\to\nu}(x)-x\|^2d\mu(x)$.  
If $\mu\ll \mathcal{L}$, by Brenier's Theorem $T_{\mu\to\nu}=\nabla\varphi$ with $\varphi$ convex and $\int_\Omega\|\nabla \varphi\|^2d\mu<\infty$, and so $v_0=\nabla u$ for  $u(x)=\varphi(x)-\|x\|^2/2$, which is also convex and with $\int_\Omega \|v_0\|^2d\mu<\infty$. 
Moreover, with the Wasserstein metric $W_2$, the space $\mathcal{P}(\Omega)$ becomes a Riemannian manifold \cite{ambrosio2005gradient,ding2021geometry,otto2001geometry,otto2000generalization}: the tangent space at a point
$\mu\in \mathcal{P}(\Omega)$ is
\begin{equation}\label{eq: tan}
  \mathrm{Tan}_{\mu}= L^2(\mu):=\left\{ v:\Omega\to\mathbb R^d: \, \|v\|^2_{L^2(\mu)}:=\int_\Omega \|v(x)\|^2 \, d\mu(x)<\infty \right\}.  
\end{equation}

\label{sec:Optimal_Transport}

\noindent\textbf{{Sliced Wasserstein Distance.}} 
Solving the optimal transport problem for dimensions bigger than one is, in general, a computationally expensive procedure. An alternative but equivalent metric to $W_2$ can be obtained by reducing the problem into several one-dimensional transport minimizations. The Sliced Wasserstein distance between probability measures $\mu$ and $\nu$ on $\mathbb{R}^d$ 
consists of first obtaining a family of one-dimensional measures through projections of $\mu$ and $\nu$ (slicing the measures), then calculating one-dimensional Wasserstein distances, and finally averaging over all the projections. 
Precisely, given $\theta\in\mathbb{S}^{d-1}$ (the unit sphere in $\mathbb R^d$), consider the projection onto the $\theta$-direction: 
\begin{equation*}
    \theta^*:\mathbb{R}^d\to\mathbb{R}, \qquad
    \theta^*(x):=\langle x,\theta\rangle ,
\end{equation*}
where $\langle\cdot,\cdot\rangle$ denotes the usual inner product in $\mathbb R^d$. Then,  the measure $\theta^*_{\#}\mu\in\mathcal{P}(\mathbb{R})$ is called the slice of $\mu$ with respect to $\theta$. This operation is related to the generalization of the Radon Transform\footnote{The Radon transform maps a function  $f\in L^1(\mathbb{R}^d)$ into a function $\mathcal{R}f\in L^1(\mathbb{R}\times \mathbb{S}^{d-1})$. For $(t,\theta)\in \mathbb{R}\times \mathbb{S}^{d-1}$,  $\mathcal{R}f(t,\theta)$ is the surface integral of $f$ over the hyperplane orthogonal to $\theta$ that passes through $t\theta$:
$\mathcal{R}f(t,\theta)=\int_{\mathbb{R}^{d-1}}f(t\theta+U_\theta \xi) \, d\xi $
where $U_\theta\in\mathbb{R}^{d\times (d-1)}$ is any matrix such that its columns form an orthonormal set of vectors in $\mathbb{R}^d$ perpendicular to $\theta$.  } $\mathcal{R}$ for measures:  $\mathcal{R}\mu\in\mathcal{P}(\mathbb{R}\times \mathbb{S}^{d-1})$ is the measure that can be \textit{disintegrated} according to slices $[\mathcal{R}\mu](\cdot,\theta) := \theta^*_\# \mu $ (see the Appendix). In particular, if
$\mu$ has density $f_\mu\in L^1(\mathbb R^d)$, the identity $\mathcal{R}f_\mu(\cdot, \theta)=\theta_{\#}^*\mu$ holds in the sense of distributions\footnote{For every test function $\psi$,
$\int_{\mathbb{R}}\mathcal{R}f_\mu(t,\theta)\psi(t) \, dt=\int_{\mathbb{R}}\int_{\mathbb{R}^{d-1}}f_\mu(t\theta+U_\theta \xi) \, d\xi \psi(t) \, dt=\int_{\mathbb{R}^{d}} f_\mu(x) \psi(\langle x,\theta\rangle) \, dx=\int_{\mathbb{R}^d}\psi(\langle x,\theta\rangle) \, d\mu=
    \int_{\mathbb{R}^d}\psi(x) \, d\theta_\#^*\mu,
$
   where in the second identity we changed variables $(t,\xi)\mapsto x\in \mathbb{R}^d$ by noticing that $t=\langle t\theta+U_\theta \xi, \theta\rangle=\langle x,\theta\rangle$.}.
Finally, the Sliced Wasserstein metric
is defined as, 
\begin{eqnarray*}
    SW_2(\mu,\nu)=\left(\int_{\mathbb{S}^{d-1}} W^2_2(\theta^*_{\#}\mu,\theta^*_{\#}\nu) \ d\theta\right)^{\frac{1}{2}},
\end{eqnarray*}
where 
$d\theta$ is the uniform measure on the sphere $\mathbb{S}^{d-1}$. For  $d=2$, we will parameterize $\mathbb{S}^1$ with angles $\theta\in[0,2\pi)$ and, in fact, it would be necessary to consider only angles in $[0,\pi]$.

\begin{remark}[Equivalences between metrics]
 In $\mathcal{P}_2(\mathbb R^d)$, $SW_2 \leq
  W_2$, and in $\mathcal{P}(B(0,R))$, where $B(0,R)$ is the ball in $\mathbb R^d$ centered at the origin $0$ with radius $R>0$,  the Wasserstein distance and the Sliced Wasserstein distances are topologically equivalent. 
\end{remark}

\section{Embeddings or Transforms}
\label{sec:embeddings_transforms}
Optimal transport has been used to establish bijective nonlinear transformations and various transport-related metrics between measures, encompassing signals, images, point clouds, and more \cite{aldroubi2021signed,bai2022sliced, bai2023linear,martin2023lcot,kolouri2015radon,kolouri2016continuous,kolouri2016sliced,wang2010optimal,wang2013linear,beier2021linear}. Despite the nuances and technicalities of each transform, the core formulation is based on the so-called Linear Optimal Transport (LOT) transform. We will introduce the \textit{recipe} to define LOT and then adapt it for different scenarios. 
\subsection{Linear optimal transport (LOT) representation}

In \cite{wang2010optimal,wang2013linear},  a linear optimal transport (LOT) framework was proposed for efficiently analyzing large databases of images. Their main contributions were the introduction of a new transport-related distance (LOT-distance) and an embedding (LOT transform) of the space of probability distributions into a linear space. The key point was to represent convoluted manifolds of images as simpler spaces. 


Let us consider a measure $\mu_r\in\mathcal{P}_2(\Omega_r)$  absolutely continuous with respect to the Lebesgue measure that we will call the \textit{reference}. 
Given a second measure $\nu\in\mathcal{P}_2(\Omega)$, there is a unique optimal Monge map $T_{\mu_r\to \nu}$. By Theorem \ref{thm: Brenier}, it is characterized as the unique map satisfying $(T_{\mu_r\to \nu})_\# \mu_r= \nu$ and  $T_{\mu_r\to \nu}=\nabla \varphi$ for some convex function $\varphi$. Thus, the application 
\begin{equation*}\label{eq: LOT application}
    \nu \to \widehat{\nu}:=T_{\mu_r\to\nu}
\end{equation*}
is a one-to-one and onto map between $\mathcal{P}_2(\Omega)$ and $L_{G}^2(\mu_r)$ that  will be called the \textit{LOT transform} (or \textit{LOT embedding}, due to its injectivity property). Its inverse transform is the application $T\to T_\# \mu_r$, which  by definition satisfies $\nu = \widehat{\nu}_\# \mu_r$.
\begin{svgraybox}
\noindent {\textbf{Recipe that defines the LOT transform}: 
\begin{itemize}
    \item Choose a reference measure $\mu_r\ll\mathcal{L}$.
    \item Given $\nu\in\mathcal{P}_2(\Omega)$, Find $T_{\mu_r\to\nu} = \nabla \varphi$ (with $\varphi$ convex) such that $(T_{\mu_r\to\nu})_\#\mu_r = \nu$, and define the LOT transform of $\nu$ as
    $\widehat{\nu}:=T_{\mu_r\to\nu}$.
    \item The inverse LOT transform is given by $T\mapsto \nu := T_\#\mu_r$.
\end{itemize}    }
\end{svgraybox}
\noindent Since $L^2_{G}(\mu_r) \subset L^2(\mu_r)$ and the LOT transform is 1-to-1, we can pull-back the metric from $L^2(\mu_r)$ to $\mathcal{P}_2(\Omega)$ and the define the \textit{LOT-distance} as 
   \begin{align*}
        d_{LOT}^{2}(\nu_1,\nu_2) &:= \|\widehat{\nu_1}-\widehat{\nu_2}\|_{L^2(\mu_r)}.
    \end{align*}



\noindent\textbf{Geometric interpretation:} Recall that the tangent space Tan$_{\mu_r}$ of the Wasserstein manifold $(\mathcal{P}(\Omega_r),W_2)$ at  $\mu_r$ is the linear space $L^2(\mu_r)$. Given any measure $\nu_i$, we can associate a vector in Tan$_{\mu_r}$ by taking the initial velocity field of the geodesic joining $\mu_r$ and $\nu_i$, which is exactly $v_0^i(x)=\widehat{\nu_i}(x) - \mathrm{id}(x)$. The LOT-distance between $\nu_1,\nu_2$ coincides with the distance between their associated vectors $v_0^1,v_0^2 \in \mathrm{Tan}_{\mu_r}$:
\begin{align}\label{eq: tangent vectors}
    \|\widehat{\nu_1}-\widehat{\nu_2}\|_{L^2(\mu_r)} &= \|(\widehat{\nu_1}-\mathrm{id}) - (\widehat{\nu_2}- \mathrm{id})\|_{L^2(\mu_r)}\notag \\
    &= \|v_0^1 - v_0^2 \|_{L^2(\mu_r)} = \|v_0^1 - v_0^2 \|_{\text{Tan}_{\mu_r}}. 
\end{align}
This motivates the name \textit{``Linear Optimal Transform''}, as the LOT-distance can be viewed as a linearized version of the OT-distance (i.e., a linearized version of the Wasserstein distance $W_2$), since it corresponds to the metric in a tangent space of the Wasserstein manifold. We refer the reader to Figure \ref{fig: tangent} for an illustration.

\begin{figure}[h!]
    \centering
    \includegraphics[width=\linewidth]{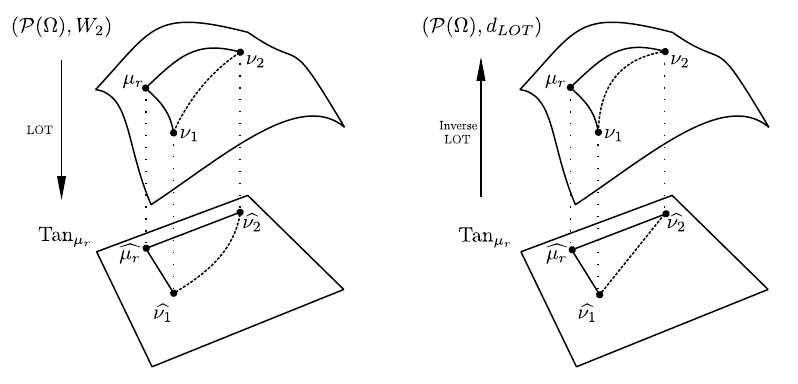}
    \caption{{Illustration of the LOT-transform, its inverse, and the LOT-distance. Let $\Omega\subset\mathbb R^d$ be convex and compact, then the LOT transform can be interpreted as the embedding of the Wasserstein manifold $(\mathcal{P}(\Omega),W_2)$ into the tangent space 
    $(\mathrm{Tan_{\mu_r}},\|\cdot\|_{L^2(\mu_r)})$
    at the \textit{reference} measure $\mu_r$. The LOT-distance is defined as the distance between vectors in such tangent space as shown in \eqref{eq: tangent vectors} ($d_{LOT}(\nu_1,\nu_2)=\|\widehat{\nu_1}-\widehat{\nu_2}\|_{{L^2(\mu_r)}}$). It follows that $d_{LOT}(\cdot,\cdot)$ preserves the Wasserstein metric when one of the measures is the reference, that is, $d_{LOT}(\mu_r,\nu_i)=W_2(\mu_r,\nu_i)$, for $i=1,2$, but in general, $d_{LOT}(\nu_1,\nu_2)\not=W_2(\nu_1,\nu_2)$. Similarly, the geodesic between $\mu_r$ and $\nu_i$ ($i=1,2$) in $(\mathcal{P}(\Omega), W_2)$ goes to a line-segment (geodesic) when applying the LOT-transform}, but the $W_2$-geodesic between general measures $\nu_1$ and $\nu_2$ is not preserved, in general, after applying LOT-transform (does not go, in general, to a line-segment in the transformed space). Reversely, line-segments (geodesics) in the LOT-transformed space (equivalently, in the tangent space $\mathrm{Tan}_{\mu_r}$) between $\widehat{\mu_r}$ an $\widehat{\nu_i}$ (for $i=1,2$) go to geodesics in the Wasserstein manifold after applying the LOT-inverse transform, but in general, line-segments (LOT-geodesics) between arbitrary `tangent vectors' $\widehat{\nu_1}$ and $\widehat{\nu_2}$ do not correspond with geodesics in the Wasserstein manifold after applying the LOT-inverse transform.}
    \label{fig: tangent}
\end{figure}

\begin{remark}\label{remark: gen LOT}
   When the reference measure $\mu_r$ is not absolutely continuous with respect to the Lebesgue measure, there is no guarantee of the existence of an optimal transport map.
    The proper extension of the LOT distance in the general case is as  \textit{the shortest generalized geodesic} connecting $\mu$ and $\nu$ \cite{ambrosio2005gradient}. We refer to\cite{wang2013linear} for the technical details.  
\end{remark}

It can be shown that the range of the LOT-transform, i.e., $L^2_G(\mu_r)$, is a convex set included in $L^2(\mu_r)$ (as convex combinations of gradients of convex functions is a gradient of a convex function). Since the latter is a geodesic space with geodesics given by `linear interpolation' (i.e., line-segments), we can copy this structure with the inverse LOT-transform. That is, with the LOT-distance, the geodesic joining $\nu_1$ and $\nu_2$ is $\rho_t:=\left((1-t)\widehat{\nu_1}+t\widehat{\nu_2}\right)_\#\mu_r$, for $0\leq t\leq 1$ which is the inverse of the geodesic (line-segment) $(1-t)\widehat{\nu_1}+t\widehat{\nu_2}$ (for $t\in[0,1]$) between $\widehat{\nu_1}$ and $\widehat{\nu_2}$ in the range of the transform.

In $\mathcal{P}_2(\Omega)$ we previously introduced the Wasserstein distance. Thus, it is natural to ask how $d_{LOT}$ compares to $W_2$. The answer is that, in general\footnote{If the reference $\mu_r$ is such that $\mu\ll\mathcal{L}$, given $\nu_1,\nu_2\in\mathcal{P}_2$, assume that that also $\nu_1\ll\mathcal{L}$ and that $T_{\mu\to\nu_1}$ is invertible, then by using the change of variables $y=T_{\mu\to\nu_1}(x)$, we have
$d_{LOT}^2(\nu_1,\nu_2)=\int \|T_{\mu\to{\nu_1}}(x)-T_{\mu\to{\nu_2}}(x)\|^2d\mu(x)=\int \|y-T_{\mu\to{\nu_2}}(T_{\mu\to{\nu_1}}^{-1}(y))\|^2d({T_{\mu\to{\nu_1}}}_\#\mu)(y)=\int \|y-T_{\mu\to\nu_2}(T_{\mu\to\nu_1}^{-1}(y))\|^2d\nu_1(y)\geq W_2^2(\nu_1,\nu_2)$, where the last inequality holds since ${(T_{\mu\to\nu_2}\circ T_{\mu\to\nu_1}^{-1})}_\#\nu_1=\nu_2$. In general, 
$d_{LOT}(\nu_1,\nu_2)\geq W_2(\nu_1,\nu_2)$ as the LOT transform comes from considering another transportation plan (as pointed in Remark \ref{remark: gen LOT}), not necessarily the optimal between $\nu_1$ and $\nu_2$.} , $d_{LOT}$ induces a finer topology  since it satisfies 
$W_2(\nu_1,\nu_2)\leq d_{LOT}(\nu_1,\nu_2)  $.
 For more details, we refer the reader to \cite{moosmuller2023linear}.


Often, the utility of different transforms is associated with how they behave under certain actions. 
The LOT transform is particularly suited for actions of certain groups of homeomorphisms of \(\mathbb{R}^d\) on the space of measures that preserve mass, i.e.,  utilizing the pushforward operation. Specifically, consider the group \( \mathcal{G}_{LOT} = \{g:\mathbb{R}^d \to \mathbb{R}^d, g(x) = ax + b: \, a\in\mathbb R_{>0}, b\in\mathbb R^d \} \) representing translations and isotropic scalings. The LOT transform translates actions in the domain of measures or the dependent variable into actions in the independent variable: The pushforward action \( g_\#\nu \) on measures is then transferred to the dependent variable by the rule \( \widehat{g_\#\nu} = g \circ \widehat{\nu} \). A clear example of this can be observed when considering $g(x) = x-\tau$ to be a pure translation. In this case, $g_\#\nu$ represents a translation in the dependant variable (i.e., a `horizontal' shift of the measure $\nu$), while its transform is $\widehat{\nu}+\tau$, that is, a translation in the independent variable (i.e., a `vertical' shift). 

A summary of the preceding discussions can be found in the following box. A similar summary will be given for the other transforms presented in this chapter. 

\begin{svgraybox}
    \textbf{PROPERTIES OF LOT TRANSFORM}
    \smallskip
    \noindent\textbf{Domain} $=\mathcal{P}(\Omega)$, where $\Omega\subset \mathbb R^d $ convex and compact. \textbf{Codomain} $=L^2(\mu_r)$.
    \smallskip
    
    \noindent\textbf{Range} $=L_G^2(\Omega_r)$.
    \textbf{Inverse}: $T\mapsto T_\#\mu_r$.
    \smallskip
    
    \noindent
    \textbf{Symmetry property}: 
    Consider the group of translations and isotropic scalings
    \begin{equation*}
      \mathcal{G}_{LOT}:=\left\{g:\mathbb R^d \to \mathbb R^d, g(x)=ax+b: \,   a\in\mathbb R_{>0}, b\in\mathbb R^d\right\}  
    \end{equation*}
    Then, for every $g\in\mathcal{G}_{LOT}$, $\nu\in\mathcal{P}(\Omega)$
                \begin{itemize}
                    \item (at measure level) $\widehat{g_\# \nu} = g\circ \widehat{\nu},$ 
                    \item (at density level) $\widehat{\mathrm{det}(Dg) \cdot (f_\nu \circ g)} = g^{-1}\circ \widehat{f_\nu}.$
                \end{itemize}
                \smallskip
    
    \noindent
    \textbf{Relation with other distances}: $d_{LOT}(\nu_1,\nu_2)\geq W_2(\nu_1,\nu_2)$ $\forall \nu_1,\nu_2\in\mathcal{P}(\Omega)$ \\ 
    and $d_{LOT}(\mu_r,\nu)= W_2(\mu_r,\nu)$ $\forall\nu\in\mathcal{P}(\Omega)$.\medskip
    
    \noindent
    \textbf{Geodesics}: $\rho_t=\left((1-t)\widehat{\nu_1}+t\widehat{\nu_2}\right)_\#\mu_r$, $0\leq t\leq 1$, is an $d_{LOT}$-geodesic between $\nu_1,\nu_2\in\mathcal{P}(\Omega)$ 
    (which is a $W_2$-geodesic if $\nu_i=\mu_r$ for some $i=1,2$). 
\end{svgraybox}

\subsection{Variants of LOT}

We will describe the variants of the LOT transform in Euclidean domains (measures defined on $\mathbb{R}^d$) or which we have a closed formula due to its relation with OT in one dimension: The Cumulative Distribution Transform (CDT), the Signed CDT (SCDT), the Radon CDT (RCDT), and the Signed RCDT (RSCDT). 

\begin{svgraybox}{ \vspace{-0.25in}
    \begin{itemize}
        \item When $d=1$, the LOT transform is called CDT.
        \item The extension of the CDT to signed measures is called SCDT.
        \item  For $d>1$, the variant of CDT that first pre-processes the signals by applying the Radon transform and then uses the CDT, is called RCDT.
        \item The extension of the RCDT to signed measures is called RSCDT.
    \end{itemize}
    \vspace{-0.25in}}
\end{svgraybox}

\noindent \textbf{One-dimensional probability measures: Cumulative Distribution Transform (CDT).}  Let $I=[a,b]$ be a real interval, and  consider $\nu \in \mathcal{P}(I)$. Following the \textit{LOT-recipe} and using \eqref{eq: gen inverse comp with cdf}, we fix continuous reference $\mu_r\in \mathcal{P}(\mathbb{R})$ (i.e., $\mu_r\ll\mathcal{L}$) and define the LOT transform of $\nu$ respect to $\mu$ and the LOT distance as 
\begin{equation}\label{eq: lot for 1d}
    \widehat{\nu} := T_{\mu_r\to\nu} = F^\dagger_\nu \circ F_{\mu_r}, 
\end{equation}
\begin{equation}\label{eq: d_lot 1d}
    d_{LOT}(\nu_1,\nu_2) = \|F^\dagger_{\nu_2} \circ F_{\mu_r} - F^\dagger_{\nu_1} \circ F_{\mu_r} \|_{L^2(\mu_r)}.
\end{equation}
The expression \eqref{eq: lot for 1d} is also called the \textit{Cumulative Distribution Transform (CDT) of $\nu$ with respect to $\mu_r$} \cite{park2018cumulative}. 
Nevertheless, by the change of variable formula, it results that the expression \eqref{eq: d_lot 1d} coincides with 
\begin{equation} \label{eq: d_cdt}
    \|F^\dagger_{\nu_2}  - F^\dagger_{\nu_1} \|_{L^2(\mathcal{L}_{[0,1]})}
\end{equation}
where $\mathcal{L}_{[0,1]}$ is the Lebesgue measure on $[0,1]$, that is, the uniform measure on $[0,1]$. In such a case, we will use the notation
$\|\cdot\|_{L^2(\mathcal{L}_{[0,1]})}=\|\cdot\|_{L^2({[0,1]})}$
This tells us that the LOT-distance in 1D is independent of the choice of the reference $\mu_r$. Moreover, the LOT transform achieves a simpler form when we use $\mathcal{L}_{[0,1]}$ as the reference. Therefore, we define the 1D-version of the LOT transform and the corresponding distance directly as follows.

\begin{definition}[CDT and CDT-distance]
    When $\nu \in \mathcal{P}(I)$, with $I=[a,b] \subset \mathbb{R}$, the LOT transform is called Cumulative Distribution Transform (CDT) and has the closed form     
    \begin{equation*}
        \widehat{\nu} := F_\nu^\dagger.
    \end{equation*}
    It is exactly the optimal Monge map from $\mathcal{L}_{[0,1]}$ to $\nu$. 
    By abuse of notation, when $\nu$ has a density $f_\nu$, the CDT transform can be defined at the density level as 
        $\widehat{f_\nu} =\widehat{\nu} = F_\nu^\dagger$.
    The CDT-distance $d_{CDT}:\mathcal{P}(I)\times\mathcal{P}(I) \to \mathbb{R}_{\geq 0}$ takes the form given in \eqref{eq: d_cdt}, that is,
    \begin{equation*}
        d_{CDT}(\nu_1,\nu_2)=\|F^\dagger_{\nu_2}  - F^\dagger_{\nu_1} \|_{L^2({[0,1]})}.
    \end{equation*}
\end{definition}

\begin{svgraybox}
    {\textbf{PROPERTIES OF CDT} 
    \smallskip
    
    \noindent\textbf{Domain} $=\mathcal{P}(I)$,  where $I=[a,b]\in \mathbb{R}$. \textbf{Codomain} $=L^2([0,1])$.
    \smallskip
    
    \noindent\textbf{Range}  $=\left\{T:[0,1]\to I  \text{ non-decreasing and right-continuous}\right\}$.
    \smallskip
    
    \noindent\textbf{Inverse}: $T\mapsto \nu := T_\#\mathcal{L}_{[0,1]}$, with $f_\nu = (T^\dagger)'$ if the density exists.\footnotemark 
    \smallskip
    
    \noindent\textbf{Symmetry property}: 
    Consider the group of homeomorphisms
    \begin{equation*}
      \mathcal{G}_{CDT}:=\left\{g:I\to I: \,  g \text{ is a strictly increasing bijection} \right\}  
    \end{equation*}
    Then, for every $g\in\mathcal{G}_{CDT}$, $\nu\in\mathcal{P}(I)$
    \begin{itemize}
        \item (at measure level) $\widehat{g_\# \nu} = g\circ \widehat{\nu},$
        \item (at density level) $\widehat{g' \cdot (f_\nu \circ g)} = g^{-1}\circ \widehat{f_\nu}.$
    \end{itemize}
    \smallskip
    
    \noindent\textbf{Relation with other distances}: $d_{CDT}(\nu_1,\nu_2) = W_2(\nu_1,\nu_2)$ $\forall \nu_1,\nu_2\in\mathcal{P}(I)$.
    \smallskip 
    
    \noindent\textbf{Geodesics}: $\rho_t=((1-t)\widehat{\nu_1}+t\widehat{\nu_2})_\#\mathcal{L}_{[0,1]}$, $0\leq t\leq 1$, is a $d_{CDT}$-geodesic between $\nu_1,\nu_2\in\mathcal{P}(I)$, which coincides with the Wasserstein geodesic.  
    }
\end{svgraybox}
\footnotetext{The derivative is in the sense of distributions. If $f_\nu$ is a pdf, then $\left(\widehat{f_\nu}^\dagger\right)^\prime=f_\nu.$}

\begin{theorem}[Characterizing Property.] \label{thm: characterization property 1d}
     If a transform $\, \widehat{\cdot}\, $ defined on $\mathcal{P}(I)$ is such that for some $\nu\ll\mathcal{L}$, we have  $\widehat{\nu}$ increasing and satisfying the symmetry property 
        \begin{equation}\label{eq: property ivan inverse}
g\circ\widehat{\nu}=\widehat{g_\#\nu} \qquad \forall g \text{ non-decreasing},         
        \end{equation}
        then $\, \widehat{\cdot} \, $ is the transform defined in \eqref{eq: lot for 1d}, that is, the CDT with respect to some reference $\mu_r\ll\mathcal{L}$. 
        $\text{[We refer the reader to the Appendix for the proof.]}$
\end{theorem}

\noindent \textbf{One dimensional signed measures. The Signed Cumulative Distribution Transform (SCDT).} Let $I=[a,b]$ be a real interval, and  consider $\nu \in \mathcal{M}(I)$ the space of signed measures over $I$. The Signed Cumulative Distribution Transform (SCDT) \cite{aldroubi2021signed} extends the CDT from $\mathcal{P}(I)$ to $\mathcal{M}(I)$. It consists of taking the Hahn-Jordan decomposition of a given measure $\nu = \nu^+ - \nu^-$, normalizing its components, and applying the CDT to each term separately. 

\begin{definition} 
    When $\nu \in \mathcal{M}(I)$ with $I=[a,b]\subset \mathbb{R}$, the SCDT is defined as
    \begin{equation*}
        \widehat{\nu} = \left(\widehat{\nu^{+N}},|\nu^+|,\widehat{\nu^{-N}},|\nu^-|\right),
    \end{equation*}
    where $\nu^{\pm N} := \nu^\pm/|\nu^\pm|$ are the normalized components of $\nu$. If $\nu^{+}\equiv 0$,  we define $\widehat{\nu}=(0,0,\widehat{\nu^{-N}},|\nu^-|)$, respectively for the case $\nu^{-}\equiv 0$, we define $\widehat{\nu}=(\widehat{\nu^{+N}},|\nu^+|,0,0)$, and if $\nu\equiv 0$, then $\widehat{\nu}=(0,0,0,0)$.  
    By abuse of notation, when $\nu$ has a `density' (i.e., Radon-Nicodym derivative) $f_\nu\in L^1(I)$,
    the transform can be defined at the density level as 
       $$ \widehat{f_\nu} =\widehat{\nu} = \left(\widehat{f_\nu^{+N}},\|f_\nu^{+}\|_1,\widehat{f_\nu^{-N}},\|f_\nu^{-}\|_1\right),$$
    where $f_\nu^{\pm N}=f_\nu^\pm/\|f_\nu^{\pm}\|_1$ and $f_\nu^+(x)=\max\{f(x),0\}$, $f_\nu^-(x)=\max\{-f(x),0\}$ and $\|f_\nu\|_1=\int|f_\nu(x)|dx$. As before, if $f_\nu^{\pm}\equiv 0$, we define $f_\nu^{\pm N}\equiv 0$ and $\widehat{f_\nu^{\pm N}}\equiv 0$.
    
    \noindent The SCDT-distance is defined as 
    \begin{align*}
        d^2_{SCDT}&(\nu_1,\nu_2) = \|\widehat{\nu_1}-\widehat{\nu_2}\|_{(L^2([0,1])\times \mathbb{R})\times (L^2([0,1])\times \mathbb{R})}^2\\
&=d_{CDT}^2(\nu_1^{+N},\nu_2^{+N}) + d_{CDT}^2(\nu_1^{-N},\nu_2^{-N}) +  \left |\nu_1^+ -  \nu_2^+\right|^2 +  \left |\nu_1^- -  \nu_2^-\right|^2\\        &=d_{CDT}^2(\nu_1^{+N},\nu_2^{+N}) + d_{CDT}^2(\nu_1^{-N},\nu_2^{-N}) +  \left( |\nu_1^+| -  |\nu_2^+|\right)^2 +  \left( |\nu_1^-| -  |\nu_2^-|\right)^2.
    \end{align*}
\end{definition}

\begin{svgraybox}
\textbf{PROPERTIES OF SCDT}\smallskip
    
    \noindent
    {\textbf{Domain} $=\mathcal{M}(I)$,  where $I=[a,b]\in \mathbb{R}$. \textbf{Codomain} $=\left(L^2([0,1]\times \mathbb R)\right)^2$.
    \smallskip
    
    \noindent\textbf{Range} $=\left\{(T,r,U,s)\in\mathcal{T}\times\mathcal{T}: \, T_\#\mathcal{L}_{|_{[0,1]}}\perp U_\#\mathcal{L}_{|_{[0,1]}} \right\}$,\\ 
    $\text{\quad \qquad where}$  $\mathcal{T}:=
        \mathrm{Range}({CDT})\times \mathbb{R}_{>0}\cup\left\{(0,0)\right\}$.\footnotemark
    \smallskip
    
    \noindent\textbf{Inverse}: $(T,r,U,s)\mapsto \nu: = r\left(T_\#\mathcal{L}_{[0,1]}\right)-s\left(U_\#\mathcal{L}_{[0,1]}\right)$\\
    $\text{\qquad \quad with}$
     $f_\nu =  r(T^\dagger)^\prime-s(U^\dagger)^\prime$ if the density exists.
    \smallskip
    
    \noindent
    \textbf{Symmetry property}: If $g\in \mathcal{G}_{CDT}$, then      
    \begin{itemize}
        \item (at measure level) $\widehat{g_\#\nu} = \left(g\circ \widehat{\nu^{+N}},|\nu^+|,g\circ \widehat{\nu^{-N}},|\nu^-|\right),$
        \item (at density level) $\widehat{g' \cdot (f_\nu \circ g)} = \left(g^{-1}\circ \widehat{f_\nu^{+N}},\|f_\nu^+\|_1,g^{-1}\circ \widehat{f_\nu^{-N}},\|f_\nu^-\|_1 \right)$.
    \end{itemize}

    \noindent
    \textbf{Relation with other distances}:
    $d^2_{SCDT}(\nu_1,\nu_2) =$
    
    \noindent$ =W_2^2(\nu_1^{+N},\nu_2^{+N}) + W_2^2(\nu_1^{-N},\nu_2^{-N}) +  \left(|\nu_1^+| -  |\nu_2^+|\right)^2 +  \left( |\nu_1^-| -  |\nu_2^-|\right)^2$.
    \medskip
    
    \noindent
    \textbf{Geodesics}: $(\mathcal{M}(I), d_{SCDT})$ is not a geodesic space (see \cite[Thm 2.5]{li2022geodesic}: as a counterexample consider $f_{\nu_1}=\mathbbm{1}_{[-1,0]}-\mathbbm{1}_{[0,1]}$, and $f_{\nu_2}=-f_{\nu_1}$). 
        However, given non-null finite \textit{positive} measures $\nu_1,\nu_2\in \mathcal{M}_{+}(I)$, then $\rho_t:=\left((1-t)|\nu_1|+t|\nu_2|\right)\left[\left((1-t)\widehat{\nu_1^N}+t\widehat{\nu_2^N}\right)_\#\mathcal{L}_{[0,1]}\right]$,
    where $\nu_i^N=\nu_i/|\nu_i|$ ($i=1,2$), is a geodesic in $\mathcal{M}_{+}(I)$ with respect to $d_{SCDT}$ (see the Appendix). }
\end{svgraybox}
\footnotetext{The first coordinate of the tuple $(0,0)$  denotes the function identically zero and the second coordinate denotes the real number zero.}

\noindent\textbf{Two-dimensional densities. The Radon Cumulative Distribution (RCDT).} 
When $\nu \in \mathcal{P}_2(\Omega)$ with $\Omega\subset \mathbb{R}^2$, to take advantage of the closed formula of the CDT, one can use the Radon transform to project to one-dimensional measures and then applying the CDT on each projection. This gives rise to the Radon Cumulative Distribution Transform (RCDT) \cite{kolouri2015radon}.


\begin{definition}
    When $\nu \in \mathcal{P}(\Omega)$ with $\Omega\subset \mathbb{R}^2$ compact,  the RCDT is defined as
    \begin{align}\label{eq: rdct}
    \widehat{\nu}^{RCDT}(\xi,\theta) &=  \widehat{[\mathcal{R}\nu] (\cdot, \theta)} (\xi)=\widehat{\theta_\#^*\nu}(\xi)=F_{\theta_\#^*\nu}^\dagger(\xi) \quad \theta \in [0,\pi], \xi\in [0,1]
\end{align}
where in the RHS of \eqref{eq: rdct} we are considering the CDT of $[\mathcal{R}\nu] (\cdot, \theta)=\theta_\#^*\nu\in \mathcal{P}(\mathbb R)$. By abuse of notation, when $\nu$ has a density $f_\nu$ the transform can be defined at the density level as
$\widehat{f_\nu}^{RCDT}(\xi,\theta)=\widehat{[\mathcal{R}f_\nu] (\cdot, \theta)} (\xi)$, where $\mathcal{R}f_\nu$ is the classical Randon transform of $L^1$-functions. 
The RCDT-distance is defined as
\begin{align}\label{eq: RCDT dist}
    d_{RCDT}^2(\nu_1,\nu_2)&=\|\widehat{\nu_1}^{RCDT}-\widehat{\nu_2}^{RCDT}\|_{L^2([0,1]\times [0,\pi])}^2\\
    &= \frac{1}{\pi}\int_{0}^\pi\int_0^1 |\widehat{\nu_1}^{RCDT}(\xi,\theta)-\widehat{\nu_2}^{RCDT}(\xi,\theta)|^2 d\xi d\theta\notag
    \end{align}
\end{definition}

For simplicity, we are considering $\Omega\subset \mathbb R^2$ compact. Thus, there exists $R>0$ such that $\Omega\subset B(0,R)$ and therefore, for each $\theta \in \mathbb S^1$, we have $\theta_\#^*\nu\in \mathcal{P}([-R,R])$. Therefore, without loss of generality,  we will assume  $\Omega =B(0,R)$ by extending any function or measure by zero outside its original domain. 



When working with the RCDT we often push measures to measures in a slice-to-slice manner. 
To do so, we start with a family of maps indexed by the angle $\theta$ that we express as a function of two variables $T(t,\theta)$. Yet, if we want to push slices of a two dimensional measure $\mu$ into slices of another two dimensional measure, we cannot do $T_\# \mu$ since the output would be a one dimensional measure.
Thus, we need to introduce a new notation: Given $T:\mathbb R\times [0,\pi] \to \mathbb R$ such that  $\forall \theta, T(\cdot,\theta)$ is  non-decreasing, we define
\begin{equation}\label{eq: LT}
        L^T (t,\theta):=\left(T(t,\theta),\theta\right)
\end{equation} 
which can be visualized in Figure \ref{fig: LT}. 
While the function $T$ above encodes how to move each slice, $L^T$ adds a second component encoding that slices remain at the same angle level. That is, 
if by abuse of notation  $
\mu(\cdot,\theta)$ denotes the slice at angle $
\theta$ (see the Appendix for the precise definition), then 
$L^T$ pushes slices $\mu(\cdot, \theta)$ to slices of $L^T_\# \mu$ by the rule $(L^T_\# \mu)(\cdot,\theta) = T(\cdot,\theta)_\# \mu(\cdot,\theta)$.

\begin{figure}[h!]
    \centering
    \includegraphics[width=0.85\linewidth]{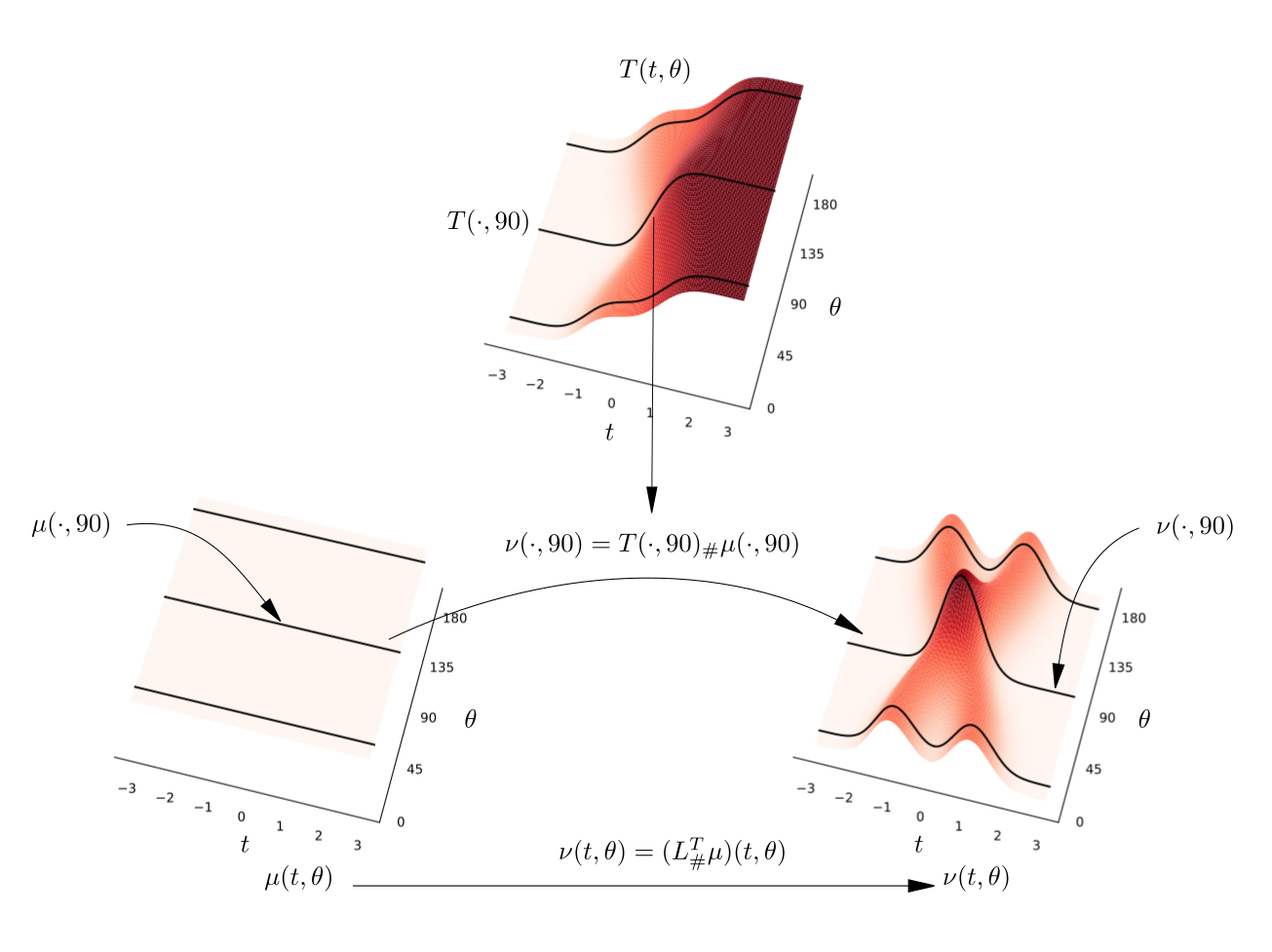}
    \caption{Visualization of the function $L^T$ defined by \eqref{eq: LT} acting as a pushforward of the measure $\mu$: For each $\theta\in[0,\pi]$, it pushes forward the $\theta$-slice of $\mu$ by $T(\cdot,\theta)$ resulting in the measure $\nu$.}
    \label{fig: LT}
\end{figure}

\begin{svgraybox}{\textbf{PROPERTIES OF RCDT}
\smallskip

\noindent\textbf{Domain} $=\mathcal{P}(\Omega)$, where $\Omega=B(0,R)\subset \mathbb R^2$. \textbf{Codomain} $=L^2([0,1]\times [0,\pi])$.\smallskip

\noindent
 \textbf{Range}  $=$ Measurable functions $T:[0,1]\times [0,\pi] \to [-R,R]$ such that $\forall \theta, T(\cdot,\theta)$ is  non-decreasing and
$L^T_\#\mathcal{L}_{|_{[0,1]\times [0,\pi]}}\in \mathrm{Range}(\mathcal{R})$.
\medskip
 

\noindent
 \textbf{Inverse}: $T\mapsto  \mathcal{R}^{-1} \left(L^T_\#\mathcal{L}_{|_{[0,1]\times [0,\pi]}}\right)$ or  
 $T\mapsto  \mathcal{R}^{-1} h$, $ \, h(\xi,\theta) := (T(\cdot,\theta)^\dagger)' (\xi)$.
 \medskip
 
\noindent  \textbf{Symmetry property}: 
Consider the group of slice homeomorphisms
    \begin{align*}
      \mathcal{G}_{RCDT}:=\{&g:[-R,R]\times [0,\pi] \to [-R,R]: \, \forall \theta\in[0,\pi],\\ 
  &g_\theta(\xi)=g(\xi,\theta) \text{ is a strictly increasing bijection} \}  
    \end{align*}
For every $g\in\mathcal{G}_{RCDT}$, $\nu\in\mathcal{P}(\Omega)$, and pdf $f_\nu$,
        \begin{itemize}
            \item (at measure level)  
            let $\nu^g =\mathcal{R}^{-1}(L^g_\#(\mathcal{R}(\nu)))$
            then,    
            \begin{equation*}
                \widehat{\nu^g}^{RCDT}(\xi,\theta) = g_\theta(\widehat{\nu}^{RCDT}(\xi,\theta))=g\left(  \widehat{\nu}^{RCDT}(\xi,\theta),\theta\right).
            \end{equation*}
            \item (at density level) let $f_\nu^g=\mathcal{R}^{-1}\left(\mathrm{det}(DL^g)\cdot\mathcal{R}(f_\nu)\circ L^g\right)$, then
            \begin{equation*}
                \widehat{f_\nu^g}^{RCDT}(\xi,\theta) = g_\theta^{-1}  \left(\widehat{f_\nu}^{RCDT}(\xi,\theta)\right).
            \end{equation*}
        \end{itemize}
        \smallskip
        
\noindent\textbf{Relation with other distances}: 
\smallskip

\noindent $d^2_{RCDT}(\nu_1,\nu_2)=\frac{1}{\pi}
    \int_{0}^\pi  d_{CDT}^2\left({\mathcal{R}\nu_1(\cdot,\theta)},{\mathcal{R}\nu_2(\cdot,\theta)}\right) d\theta= SW_2^2(\nu_1,\nu_2)$.
    
    }
\end{svgraybox}

\begin{remark}[Translations and isotropic dilations]
    In particular, when working with translations and anisotropic dilations, the symmetry property reads as follows:
    \begin{itemize}
        \item Let $b\in\mathbb{R}^2$ be a shift parameter, and consider $\tau_b(x):=x+b$ for all $x\in\mathbb R^2$, then
    $$\widehat{(\tau_b)_\#\nu}^{RCDT}(\xi,\theta)=\widehat{\nu}^{RCDT}(\xi,\theta)+ \theta^*(b).$$
    \item Let $a\in\mathbb{R}_{>0}$ be a scaling parameter, and consider $D_a(x):=ax$ for all $x\in\mathbb R^2$ and $d_a(\xi):=a\xi$ for all $\xi\in\mathbb R$, then
    $$\widehat{(D_a)_\#\nu}^{RCDT}=d_a\circ\widehat{\nu}^{RCDT}.$$
    \end{itemize}
    
\end{remark}

\begin{remark}\label{remark: rcdt not depend ref}
In the literature (see, for e.g., \cite{kolouri2015radon,gong2023radon}), the definition of the RCDT and the RDCT-distance is often written in a more general way, by taking a reference measure $\mu_r\in\mathcal{P}(B(0,R))$. Such a new distance coincides with the given RCDT-distance in \eqref{eq: RCDT dist}. See the Appendix for a detailed discussion.   
\end{remark}

\noindent\textbf{Two-dimensional signed measures. The Radon Signed Cumulative Distribution Tranform (RSCDT).} 
For simplicity, consider $\Omega=B(0,R)$ the ball centered at zero and with radius $R$ in $\mathbb R^2$. The Radon Signed Cumulative Distribution Transform (RSDCT) \cite{gong2023radon} extends the RCDT from $\mathcal{P}(\Omega)$ to $\mathcal{M}(\Omega)$. It consists of taking the Radon transform and then SCDT for each projection angle.

\begin{definition}
    When $\nu \in \mathcal{M}(\Omega)$ with $\Omega\subset \mathbb{R}^2$ compact,  the RSCDT is defined as
    \begin{align*}\label{eq: rsdct}
        \widehat{\nu}^{RSCDT}&(\xi,\theta) =  \widehat{[\mathcal{R}\nu] (\cdot, \theta)} (\xi)=\widehat{\theta_\#^*\nu}(\xi)\\
        &=\left( \widehat{[\mathcal{R}\nu] (\cdot, \theta)^{+N}}(\xi),|[\mathcal{R}\nu] (\cdot, \theta)^{+}| ,\widehat{[\mathcal{R}\nu] (\cdot, \theta)^{-N}}(\xi) , |[\mathcal{R}\nu] (\cdot, \theta)^{-}|\right)\\
        &=\left( \widehat{[\theta_\#^*\nu]^{+N}}(\xi),|[\theta_\#^*\nu] ^{+}| ,\widehat{[\theta_\#^*\nu]^{-N}} (\xi), |[\theta_\#^*\nu] ^{-}|\right) \quad \forall \theta \in [0,\pi], \xi\in [0,1]
    \end{align*}
where the $ \, \widehat{\cdot}\, $ used in the first equality is the SCDT, while in the last equality, it is the CDT\footnote{We note that, even though $    \theta_\#^*\nu=\theta_\#^*(\nu^+)-\theta_\#^*(\nu^-)$, it doesn't hold that $\theta_\#^*(\nu^\pm)=(\theta_\#^*\nu)^\pm$ since the projections $\theta_\#^*(\nu^+)$ and $\theta_\#^*(\nu^-)$ are not necessarily mutually singular.}. 
As for the SCDT, if $[\theta_\#^*\nu]^\pm=0$, then we define $[\theta_\#^*\nu]^{\pm N}=0$ and $\widehat{[\theta_\#^*\nu]^{\pm N}}=0$.
By abuse of notation, when $\nu$ has a `density' (i.e., a Radon-Nicodym derivative) $f_\nu\in L^1(\Omega)$ the transform can be defined at the density level as
\begin{align*}
  &\widehat{f_\nu}^{RSCDT}(\xi,\theta)=\widehat{[\mathcal{R}f_\nu] (\cdot, \theta)} (\xi)\\
  &=\left( \widehat{[\mathcal{R}f_\nu] (\cdot, \theta)^{+N}},\|[\mathcal{R}f_\nu] (\cdot, \theta)^{+N}\|_{L^1(\mathbb R)} ,\widehat{[\mathcal{R}f_\nu] (\cdot, \theta)^{-N}} , \|[\mathcal{R}f_\nu] (\cdot, \theta)^{-N}\|_{L^1(\mathbb R)}\right),  
\end{align*}
where $\mathcal{R}f_\nu$ is the usual Radon transform defined for $L^1$-functions.

\noindent The RSCDT-distance is defined as
\begin{align*}
    &d_{RSCDT}^2(\nu_1,\nu_2)
    =\|\widehat{\nu_1}^{RSCDT}-\widehat{\nu_2}^{RSCDT}\|_{(L^2([0,1]\times [0,\pi]))^4}^2\\
    &= \frac{1}{\pi}\int_{0}^\pi\int_0^1 |\widehat{\nu_1}^{RSCDT}(\xi,\theta)-\widehat{\nu_2}^{RSCDT}(\xi,\theta)|^2 d\xi d\theta\notag\\
    &= \frac{1}{\pi}\int_{0}^\pi\int_0^1 |\widehat{[\theta_\#^*\nu_1]^{+N}}(\xi)-\widehat{[\theta_\#^*\nu_2]^{+N}}(\xi)|^2 +|\widehat{[\theta_\#^*\nu_1]^{-N}}(\xi)-\widehat{[\theta_\#^*\nu_2]^{-N}}(\xi)|^2 d\xi d\theta  \notag\\
    &\quad+ \frac{1}{\pi}\int_{0}^\pi\int_0^1 |[\theta_\#^*\nu_2]^{+}-[\theta_\#^*\nu_2]^{+}|^2+ |[\theta_\#^*\nu_2]^{-}-[\theta_\#^*\nu_2]^{-}|^2d\theta d\xi .
    \end{align*}
\end{definition}

\noindent Given a non-negative function $r:[0,\pi]\to\mathbb R_{\geq 0}$, let us denote by $\mathcal{L}_{|_{[0,1]\times [0,\pi]}}^{r}$ the measure supported on $[0,1]\times [0,\pi]$ with density $(t,\theta)\mapsto r(\theta)$. This and the notation used for RCDT allow us to summarize the properties of the RSCDT as follows.

\begin{svgraybox}{\noindent \textbf{PROPERTIES OF RSCDT }\smallskip

\noindent \textbf{Domain} $=\mathcal{M}(\Omega)$,  $\Omega=B(0,R)\subset \mathbb R^2$. \textbf{Codomain} $=\left(L^2([0,1]\times[0,\pi])\right)^4$.
\smallskip

\noindent \textbf{Range}  $=$
4-tuple $(T,r,U,s)$ where $(T,r)$ is a pair of measurable functions $T:[0,1]\times [0,\pi]\to [-R,R]$, $r:[0,\pi]\to \mathbb R_{\geq 0}$  such that for each $\theta$
$ \left[T(\cdot,\theta) \text{ is non-decreasing and } r(\theta)>0\right]
             \text{or } \left[T(\cdot,\theta)=0 \text{ and } r(\theta)=0\right]$ (analogously for $(U,s)$), with the conditions: $\forall\theta$, $(T(\cdot,\theta))_\#\mathcal{L}_{|_{[0,1]}}\perp (U(\cdot,\theta))_\#\mathcal{L}_{|_{[0,1]}}$ and also $L^T_\#\mathcal{L}_{|_{[0,1]\times [0,\pi]}}^r- L^U_\#\mathcal{L}_{|_{[0,1]\times [0,\pi]}}^s\in \mathrm{Range}(\mathcal{R})$.
             
\medskip

\noindent \textbf{Inverse}: $(T,r,U,s)\mapsto  \mathcal{R}^{-1} \left( L^T_\#\mathcal{L}_{|_{[0,1]\times [0,\pi]}}^r- L^U_\#\mathcal{L}_{|_{[0,1]\times [0,\pi]}}^s\right)$ or 
\smallskip

\noindent $(T,r,U,s)\mapsto  \mathcal{R}^{-1} (h)$,  $h(\xi,\theta) :=r(\theta)(T(\cdot,\theta)^\dagger)' (\xi) -s(\theta)(U(\cdot,\theta)^\dagger)' (\xi)$.
\medskip

\noindent
  \textbf{Symmetry property}: If $g\in \mathcal{G}_{RCDT}$, then 
        \begin{itemize}
            \item (at measure level) given $\nu\in\mathcal{M}(\Omega)$, 
            let $\nu^g =\mathcal{R}^{-1}(L^g_\#(\mathcal{R}(\nu)))$ and
            assume $\widehat{\nu}^{RSCDT}(\xi,\theta) = (T(\xi,\theta),r(\theta),U(\xi,\theta),s(\theta))$, then 
            \begin{equation*}
                \widehat{\nu^g}^{RSCDT}(\xi,\theta) = (g(T(\xi,\theta),\theta),r(\theta),g(U(\xi,\theta),\theta),s(\theta)).
            \end{equation*}
            \item (at density level) given a $f_\nu\in L^1(\Omega)$, $f_\nu^g=\mathcal{R}^{-1}\left(\mathrm{det}(DL^g)\cdot\mathcal{R}(f_\nu)\circ L^g\right)$ and
            assume $\widehat{f_\nu}^{RSCDT}(\xi,\theta) = (T(\xi,\theta),r(\theta),U(\xi,\theta),s(\theta)),$ then
            \begin{equation*}
                \widehat{f_\nu^g}^{RSCDT}(\xi,\theta) = (g_\theta^{-1}(T(\xi,\theta)),r(\theta),g_\theta^{-1}(U(\xi,\theta)),s(\theta)).
            \end{equation*}
        \end{itemize}

\noindent
\textbf{Relation with other distances}: 
\smallskip

\noindent$d^2_{RSCDT}(\nu_1,\nu_2)=\frac{1}{\pi}\int_0^\pi d^2_{SCDT}(\mathcal{R}\nu_1(\cdot,\theta), \mathcal{R}\nu_1(\cdot,\theta)) \ d\theta$.
    }
\end{svgraybox}

\subsection{Numerical methods and software}

Numerical methods for computing the Monge and Kantorovich versions, based on PDEs and linear programming have existed for decades (an excellent repository can be found in the Python Optimal Transport (POT) library \cite{flamary2021pot}). A detailed description of these is omitted for brevity and is beyond the scope of this document. We do note that continuous formulations based on the solution to the Monge-Ampere equations, as well as constrained minimization methods, exist. For the discrete problem, traditional methods based on linear programming have recently been enhanced through the use of entropy-based regularization methods, facilitating a quick solution. A more complete discussion can be found in \cite{cuturi2013sinkhorn, kolouri2017optimal,peyre2019computational}. As far as the transport transforms discussed above derived from the LOT concept, software for computing the discrete \cite{wang2013linear} and continuous \cite{kolouri2016continuous} versions is available in the PyTransKit software package \cite{pytranskit}. Software for computing the CDT, SCDT, RCDT, RSCDT using numerical integration is also available in the PyTranskit package \cite{pytranskit}.

\section{Applications}\label{sec: app}


The transport-based signal transformation framework has been used to generate state of the art results in numerous applications in data classification  \cite{park2018multiplexing,neary2021transport,shifat2021radon,kundu2018discovery,basu2014detecting,rubaiyat2021nearest,liao2021multi,guan2019vehicle,lee2021identifying,kolouri2020wasserstein,kundu2020enabling,ozolek2014accurate, tosun2015detection,roy2020recognizing,shifat2020cell,zhuang2022local,thorpe2017transportation,rabbi2022invariance,kolouri2016radon,kolouri2016continuous,kolouri2016sliced,kolouri2018sliced,kolouri2019generalized,kolouri2017optimal}, as well as signal estimation \cite{rubaiyat2020parametric,nichols2019time}, communications \cite{park2018multiplexing,neary2021transport}, image formation in turbulence \cite{nichols2018transport}, optics \cite{nichols2019transport,nichols2021vector}, particle and high energy physics \cite{cai2020linearized}, system identification \cite{rubaiyat2024data}, structural health monitoring \cite{wang2020fault,zhang2024combining,rubaiyat2020parametric}, reduced order modeling solutions of PDEs \cite{ren2021model}, and others \cite{basu2014detecting,rahimi2017automatic,cloninger2023linearized}. Here we focus our description on three broad areas: signal estimation, classification, and transport-based morphometry.

\subsection{Linearizing estimation problems}
Numerous applications (e.g. radar, communications) require us to estimate a certain signal characteristic or parameter that provides information about the physical environment. Typically, signal parameters of interest include amplitude, frequency, time delay (or spatial shift in the case of images). Consider for example the radar problem where a known signal $f(t)$ (an electromagnetic waveform) is emitted aimed at a moving object. The object reflects the wave and the received signal is of the form
\[ h(t)=\alpha f(\omega t - \tau) + \eta(t) \]
where $\eta(t)$ is a noise process, $\alpha$ refers to an arbitrary (i.e., calibration) amplitude parameter, $\omega$ a frequency shift, and $\tau$ a certain time delay. Recovering the unknown parameters $\omega, \tau$ would give us information about the object's velocity and position (the same principle is used to localize sources, such as a cell phone), and it is what is known as an estimation problem. 

We can note that the parameter $\alpha$ is unrelated to the quantities of interest (position and speed of the object), and can be arbitrary (i.e., affected by the material properties of the reflecting object, calibration of the system, and so on). One way to deal with it so as not to affect the analysis is to simply work with the normalized energy density of the signal $h^2/\|h\|_2^2$ (and normalizing the known template $f$ the same way) as discussed in \cite{nichols2019time,rubaiyat2020parametric}. 
Assuming both $f$ and $h$ have been converted into energy densities, we can now rewrite the estimation problem as searching for the transformation  $g(t) = \omega t-\tau$ such that $f^g(t) = g'(t) f (g(t)) =\omega f(\omega t-\tau ) \sim h(t)$. 
To do so, it is common to establish a metric $d(\cdot,\cdot)$ between the received signal and parametric model  and define the solution to the estimation problem as 
\[ \min_g d^2(f^g,h)\]
When $d^2(f^g, h) = \| f^g - h\|_2^2$ this problem is typically nonlinear and non-convex. 
Numerous approaches for finding the solution to such estimation problems have been proposed including cross-correlation, ESPRIT, MUSIC, and others (see, for example, \cite{jakobsson1998subspace}). 
Another commonly used approach is to locate the peak of Wide-band Ambiguity Function (WBAF) $A_{f^g,h}(\omega, \tau)$ between the measured signal $h$ and the known signal $f$  \cite{jin1995estimation,niu1999wavelet,tao2008two}, 
which is nonlinear, non-convex, and thus non-trivial to optimize.

As an alternative, 
\cite{rubaiyat2020parametric,nichols2019time} proposed to utilize the Wasserstein metric with the aid of the Cumulative Distribution Transform \cite{park2018cumulative}. The idea is to note that, using the symmetry property of the CDT, 
$$d^2_{CDT}(f^g,h) = \| \widehat{f^g} - \widehat{h}  \|_2^2 = \| g^{-1}\circ\widehat{f} - \widehat{h}  \|_2^2=\left\|\frac{1}{\omega}\widehat{f}+\frac{\tau}{\omega}-\widehat{h}\right\|_2^2.$$ Renaming 
$a=1/\omega$, $b=\tau/\omega$, we can optimize for
$$\min_{a,b}\|a\widehat{f}+b-\widehat{h}\|_2^2$$
which we can note is a quadratic function on the unknown parameters $(a,b)$. We can thus minimize it with the well-known \textbf{linear} least squares technique, greatly simplifying the procedure. References \cite{nichols2019time,rubaiyat2020parametric} have extensive analysis with respect to noise and comparison between numerous techniques. Figure \ref{fig:estimation_fig} gives an overview of the simplification.

\begin{figure}[h!]
    \centering
    \includegraphics[width=0.65\linewidth]{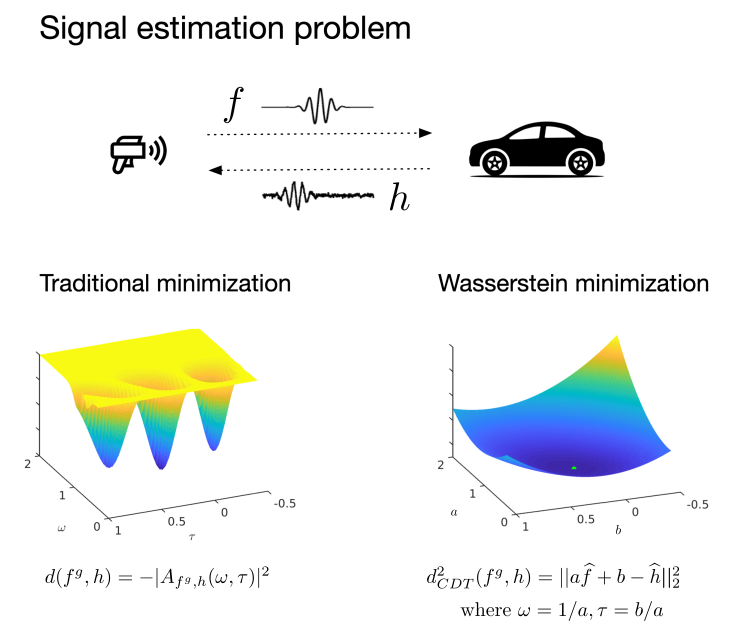}
    \caption{Depiction of a typical parametric signal estimation problem. A frequency and time delay problem can be converted into a linear least squares minimization problem by using transport transforms (CDT).}
    \label{fig:estimation_fig}
\end{figure}

\subsection{Linearizing classification problems}

The symmetry property of transport transforms helps identify which types of classification problems each transform can solve. The analysis comes from understanding how classes are generated and represented in transform space.  Generally, classes consist of deformations of a template distribution $\nu$, forming sets $\mathbb{S}_{\nu,G} = \{\nu^g: \,  g \in G\}$, where $\nu^g$ and $G$ depend on each specific transport transform framework, as summarized in Table \ref{tab: table classes}. 
These models for classes are often called in the literature as \textit{algebraic generative models}. 
When $G$ is a convex group, the classes will be convex in transform space as they are essentially conformed by $g \circ \widehat{\nu}$ such that $g \in G$. Roughly speaking, changes in the independent variable in the original signal space transform into deformations in the dependent variable. In this case, when working with finite samples, any two distinct classes can be linearly separated in the transform space. We refer the reader mainly to \cite{aldroubi2021partitioning, kolouri2016radon}.


\begin{table}[h]
\centering
\caption{}
\label{tab: table classes}       
%
%
\begin{scriptsize}
\begin{tabular}{p{1.4cm}p{2.75cm}p{4.35cm}p{0.9cm}p{1.5cm}}
    \hline\noalign{\smallskip}
    $\nu$ & $\nu^g$& density $f_\nu^g$ &$\widehat{\cdot}$ & $G$ \\
    \noalign{\smallskip}\svhline\noalign{\smallskip}
    $\nu\in\mathcal{P}(\mathbb R)$& $\nu^g=g_\#\nu$& $f_\nu^g=g'\cdot f_\nu\circ g$ &CDT&
    $G\subseteq\mathcal{G}_{CDT}$ \\
    \hline
    $\nu\in\mathcal{P}(\mathbb R ^2)$& $\nu^g=\mathcal{R}^{-1}(L^g_\#(\mathcal{R}(\nu)))$ &$f_\nu^g=\mathcal{R}^{-1}\left(\mathrm{det}(DL^g)\cdot\mathcal{R}(f_\nu)\circ L^g\right)$ &RCDT & $G\subseteq \mathcal{G}_{RCDT}$\\
    \hline
    $\nu\in\mathcal{P}(\mathbb R^d)$& $\nu^g=g_\#\nu$ & $f_\nu^g=\mathrm{det}({Dg})\cdot f_\nu\circ g$ &LOT&
    $G\subseteq\mathcal{G}_{LOT}$ \\
    \noalign{\smallskip}\hline\noalign{\smallskip}
\end{tabular}
\end{scriptsize}
\end{table}

\begin{theorem} By using the notation in each row of Table \ref{tab: table classes}, we have the following:  
    Given a template $\nu$ with compact support, consider the class
    $C_\nu=\{\nu^g: \, g\in G\}$, where $G$ is a convex subgroup of $\mathcal{G}$, then
    $\widehat{\mathbb{S}_{\nu,G}}=\{\widehat{\nu^g}: \, g\in G\}$
    is a convex set in the transport transform space.
\end{theorem}

When working with real-world classification problems, verifying that classes satisfy the algebraic generative model given above is often not possible. The heuristic is that if elements in the same class can be considered as some mass-preserving deformation of a template, then transport transforms could be used. As an example, consider classifying images of segmented faces
into faces that are either neutral (class 1) or smiling (class 2). Figure \ref{fig:classification_fig_1} shows a few example faces from each class, along with the projection of previously separated test data onto the 2D penalized linear discriminant (PLDA) analysis \cite{wang2011penalized} subspace computed from training data (see \cite{kolouri2016radon} for more details). In the right panels, each dot corresponds to the projection of one image. The top plot corresponds to PLDA projections computed using the native signal intensities, while the bottom plot shows the projections computed in RCDT \cite{kolouri2016radon} space. The RCDT space offers a simplified data geometry where machine learning can be more successful.

\begin{figure}[h!]
    \centering
    \includegraphics[width=0.95\linewidth]{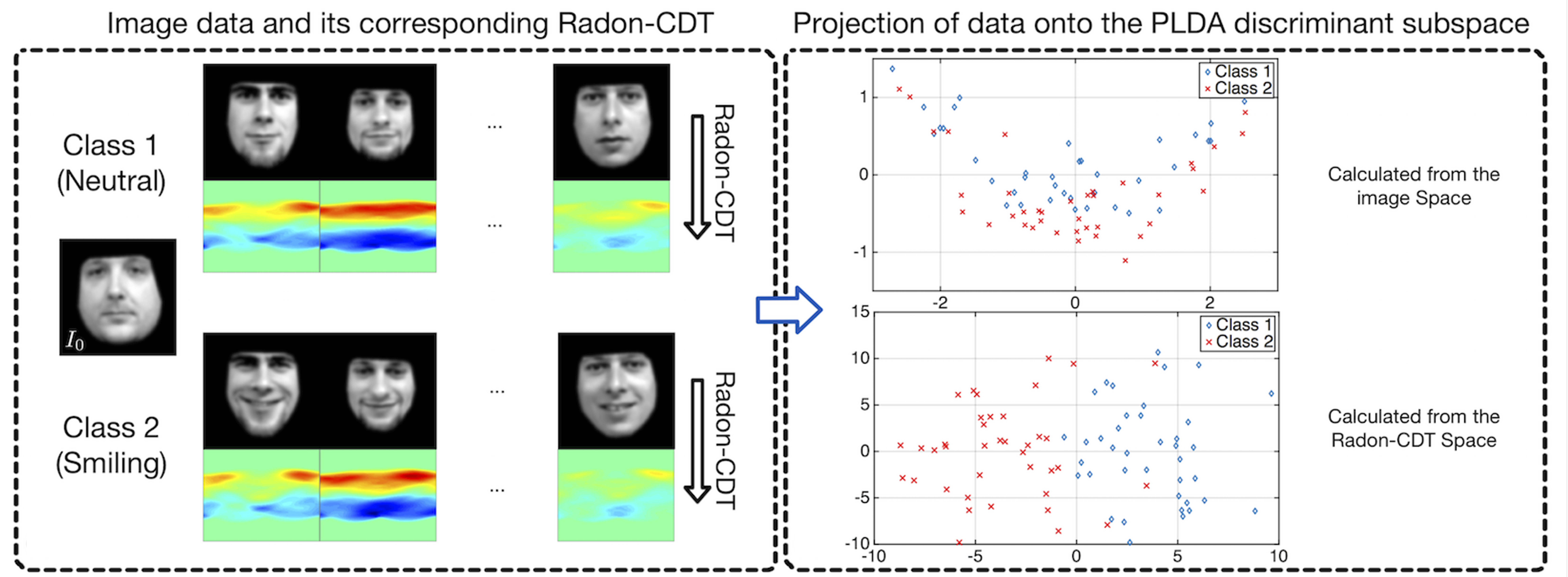}
    \caption{Classifying data in native image domain as opposed to in transport transform domain (RCDT). Smiling and neutral faces are non-linearly distinguishable from one another in the image domain. The data classes become more convex in the transform domain, where they can be separated with a linear classifier.  Figure adapted from \cite{kolouri2016radon}.}
    \label{fig:classification_fig_1}
\end{figure}

As another example, we demonstrate the ability of the transport signal transformation framework to create a classifier capable of dealing with channel nonlinearities due to turbulence in optical communications. A simple nearest subspace classification method in RCDT domain \cite{shifat2021radon} can be used to effectively classify symbols that represent sequences of 5 bits (see Figure \ref{fig:classification_fig_2} Top-Right). 
In this case, the classes consist of noisy deformations of template symbols that occur during transmission over an optical channel as represented in Figure \ref{fig:classification_fig_2} Top-Left. 
The transport-based classification method can show substantial improvements with respect to state-of-the-art Deep Learning in classification accuracy (bit error rate) as a function of the amount of training data used (Fig. \ref{fig:classification_fig_2} bottom plot 1), out-of-distribution performance (Fig. \ref{fig:classification_fig_2} bottom plot 1 plot 2) and computational performance (Fig. \ref{fig:classification_fig_2} bottom plot 3), where several orders of magnitude improvements are obtained). 

\begin{figure}[h!]
    \centering
    \includegraphics[width=\linewidth]{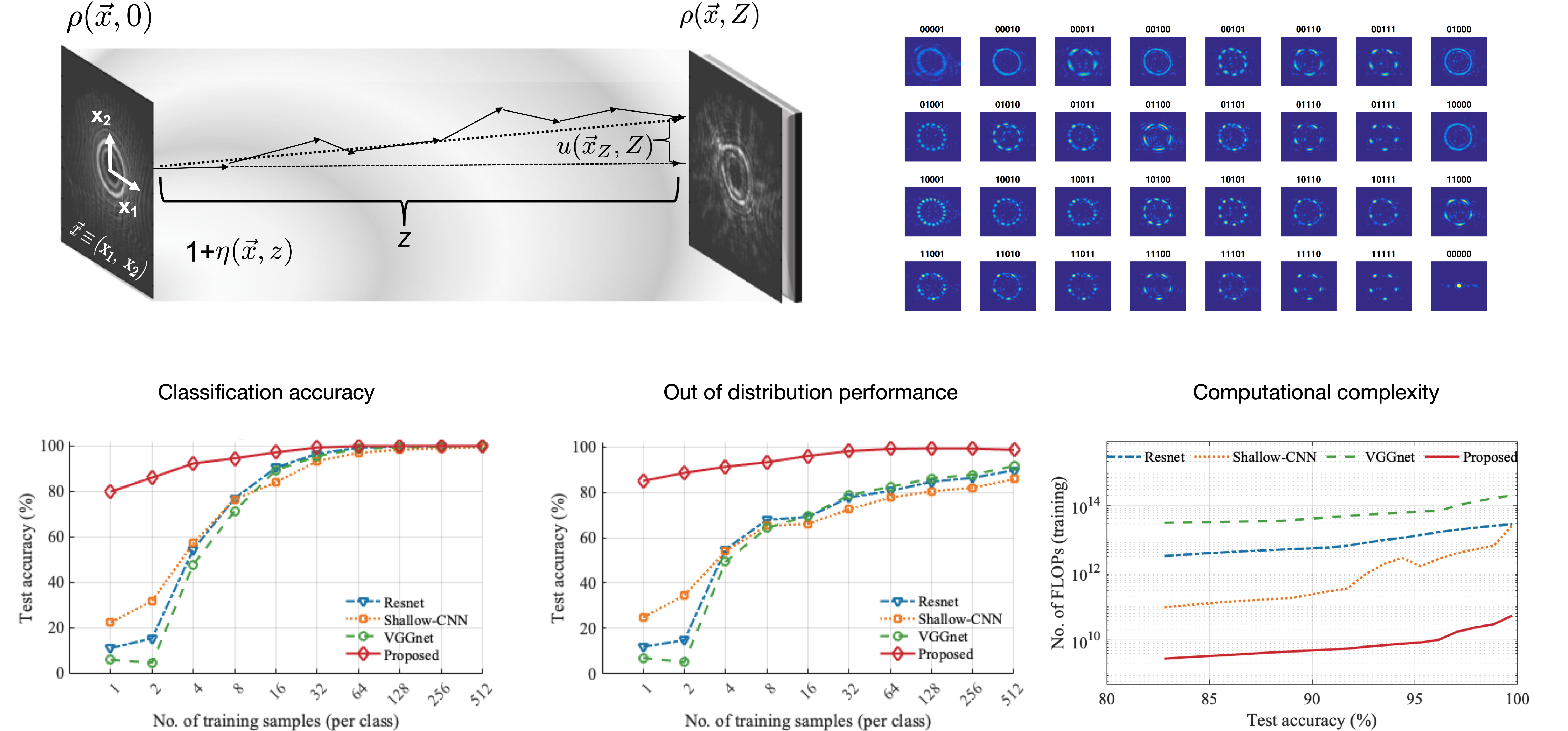}
    \caption{OAM patterns (top-right)  are emitted by the transmitter, travel through a turbulent channel, and must be decoded by the receiver (32 class classification problem). The transport systems representation (red) substantially outperforms deep learning techniques in terms of classification accuracy as a function of the number of training data, out-of-distribution performance, and number of floating points operations (FLOPS). Figure adapted from \cite{shifat2021radon}.}
    \label{fig:classification_fig_2}
\end{figure}

\subsection{Transport-based Morphometry (TBM)}

Beyond facilitating signal estimation and classification problems by rendering data classes convex, the transport transformation framework can also be viewed as a more general mathematical modeling technique. When applied to modeling forms, structures, or morphology, the technique has been named Transport-based Morphometry (TBM) \cite{basu2014detecting,kundu2018discovery,kundu2020enabling,kundu2024discovering,ironside2024fully}. 

In medicine and biology, TBM has been applied to model subcellular mass distributions \cite{basu2014detecting,wang2010optimal,ozolek2014accurate,tosun2015detection,hanna2017predictive,rabbi2023transport}, gray and white matter brain morphology \cite{kundu2018discovery,kundu2024discovering,kundu2019assessing,kundu2021investigating},  brain hemorrhage morphology in stroke \cite{ironside2024fully}, and knee cartilage in osteoarthritis \cite{kundu2020enabling}, to name a few. The idea is simple and can be interpreted in a 3-step pipeline as shown in Figure \ref{fig:tbm_fig_1}. Given a set of labeled data (signals or images) $\{f_1,y_1\},\{f_2,y_2\},\cdots,\{f_N,y_N\}$, where $f_k$ corresponds to the signal and $y_k$ to a label such as patient age, molecular measurement, class (diseased vs. non-diseased), we first choose a suitable reference (such as the intrinsic mean of the dataset as suggested in \cite{basu2014detecting,kolouri2016continuous}) and use it to compute the transport transform (LOT, CDT, RCDT, etc.) to obtain a representation in transport domain $\{\widehat{f}_1,y_1\},\{\widehat{f}_2,y_2\},\cdots,\{\widehat{f}_N,y_N\}$. 
Once we have a representation in transport transform space,  different statistical methods (PCA, LDA, PLDA, CCA, etc) can be employed to understand the salient features in the dataset, the most important attributes associated with each label or class, or other characteristics of interest. Interestingly enough, the characterizations obtained in transform space can often be traced back to physical quantities such as form or shape providing a \textit{natural} view of the dataset.

\begin{figure}[h!]
    \centering
    \includegraphics[width=\linewidth]{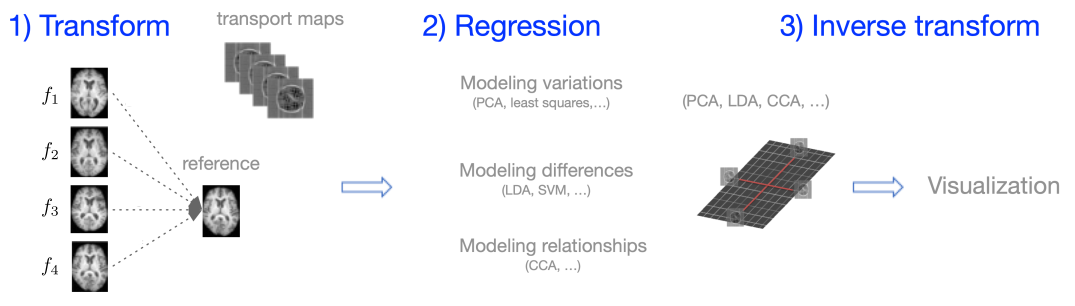}
    \caption{Transport-based morphometry (TBM) can be understood as a three-step process. In the first step, a reference measure is chosen, and the transport transforms of the signals/images are computed. In the second step,  the data represented in transform space is modeled according to different goals. The third step consists of visualizing the model obtained by the framework by taking the inverse transforms of the representations in the previous step.} 
    \label{fig:tbm_fig_1}
\end{figure}

For example, given a large set of images $f_1,f_2,\cdots,f_N$ scientists are often interested in organizing the data into its main principal components - i.e., a decomposition of different modes describing the variation in forms in the dataset. Principal component analysis (PCA) \cite{anderson1958introduction} is a common way to do this, and has been extensively employed in medical image analysis to visualize `registration' modes (see \cite{frangi2002automatic} for example). 

Though the procedure can be described in the continuous sense, here we describe its discrete implementation, and thus we take $\widehat{f}_k$ to be (sampled) discrete vectors. 
The idea in PCA is to compute a set of orthonormal basis functions $\{\widehat{\phi_k}\}_{k=1}^M$ which can reconstruct the given dataset in the least squares sense. 
Since we apply PCA in the transport transform space, the procedure involves computing the covariance matrix of the transformed dataset $S = 1/N \sum_{k=1}^{N} (\widehat{f}_k-\bar{m})(\widehat{f}_k-\bar{m})^T$, where $\bar{m} =  1/N \sum_{k=1}^{N} \widehat{f}_k$ is its mean. 
It can be shown that the top $M<N$ eigenvectors $\widehat{\phi_k}$ of $S$ correspond to the solution of the least squares approximation problem, subject to the constraint that the set $\{\widehat{\phi_k}\}_{k=1}^M$ be orthonormal. 
Each $\widehat{\phi_k}$ can be thought to encode a transport `mode of variation' explaining the dataset. 
To visualize the deformation mode $\widehat{\phi_k}$ we can invert the line $\widehat{f_{\alpha,k}} :=\bar{m} + \alpha\widehat{\phi_k}$, $\alpha\in\mathbb R$, in transform space 
to obtain the respective functions $f_{\alpha,k}$ in native space.
This procedure is demonstrated in a nuclear structure dataset in Figure \ref{fig:tbm_fig_PCA} (see \cite{basu2014detecting,rabbi2023transport} for more details). In this database, each image corresponds to a nucleus segmented from a microscopy image. The intensity at each pixel is proportional to the relative amount of chromatin present at that location. The left panel shows the raw images and the right panel shows the deformation modes (obtained by the TBM method) explaining shape and texture variations present in the dataset. The first mode corresponds to size differences, the second to elongation, the third chromatin displacements towards the northeast/southwest direction, the fourth contains variations in chromatin from the center of the nucleus to the periphery in a radial fashion, and so on. These help us visualize and understand the variations present in the dataset in terms of language that can be associated with specific physical characteristics. 

\begin{figure}[h!]
    \centering
    \includegraphics[width=0.95\linewidth]{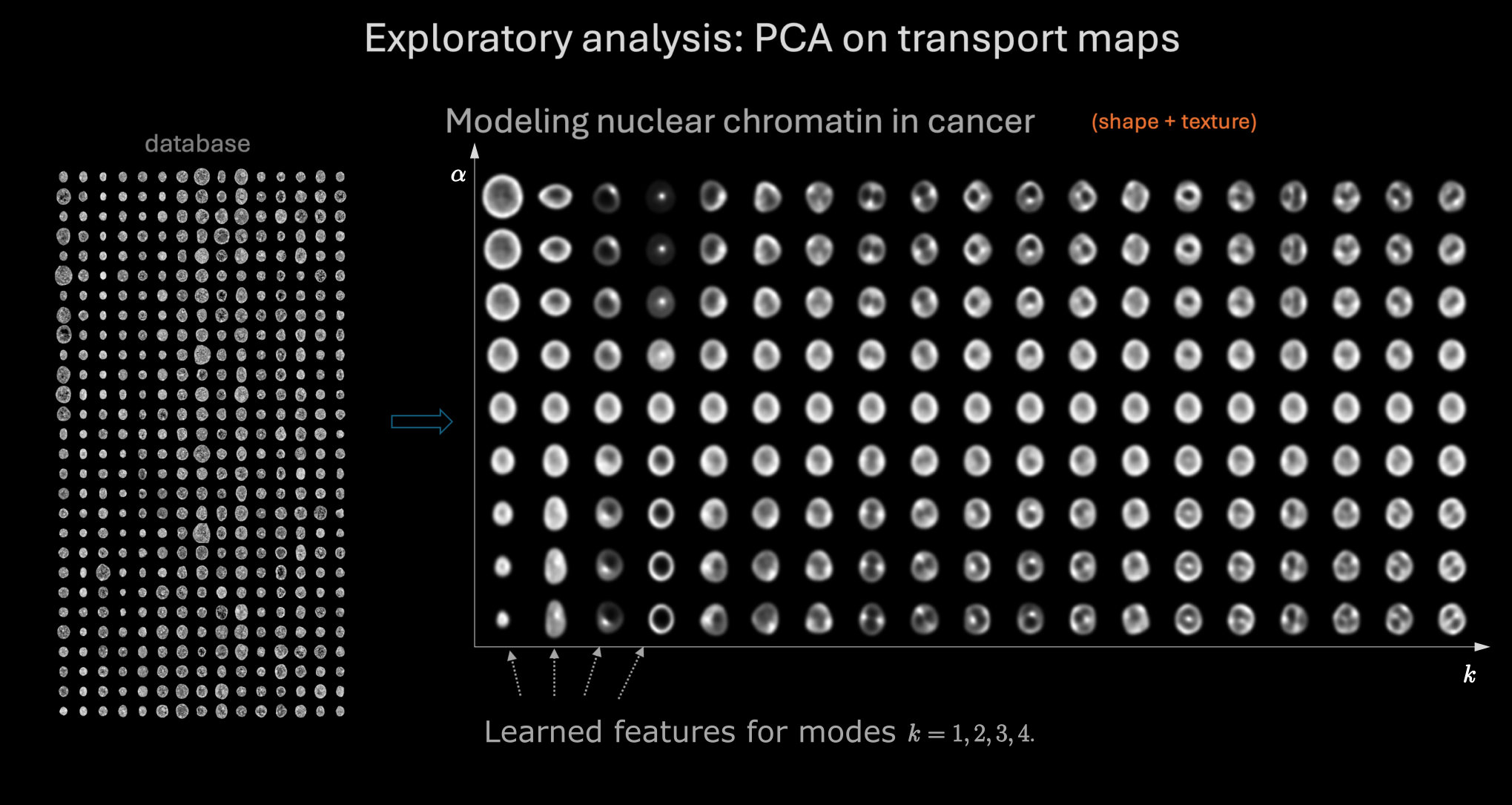}
    \caption{Transport-based morphometry (TBM) used to visualize variations in texture and shape of nuclear chromatin in cancer.}
    \label{fig:tbm_fig_PCA}
\end{figure}

In addition to visualizing the main modes of variation within a dataset, we can also use TBM to visualize the main modes of \emph{discrimination} between different classes. Let the dataset be as defined above $\{f_1,y_1\},\{f_2,y_2\},\cdots,\{f_N,y_N\}$, where the labels $y_k$ contain class information (e.g. malignant vs. benign cancer). We can use the TBM processing pipeline to discover modes of variation that characterize differences between these classes. 
Indeed, we can use regression/classification methods on $\{\widehat{f}_1,y_1\},\{\widehat{f}_2,y_2\},\cdots,\{\widehat{f}_N,y_N\}$ to discover meaningful differences between the classes. One such procedure could be the so-called Penalized Linear Discriminant Analysis (PLDA) method \cite{wang2011penalized}, which is an improvement upon the classic Linear Discriminant Analysis (LDA). 
The new directions $\widehat{\phi_k}$ obtained  by PLDA are characterized by the solution of the generalized eigenvalue problem $S \widehat{\phi} = \lambda (S_W + \gamma I)\widehat{\phi}$, 
with $S$ as above in the PCA procedure, and where $S_W$ is the so-called `within' covariance matrix which is simply the sum of the covariance matrices corresponding to each class. The procedure can be understood as a `mixture' between PCA and LDA, with $\gamma$ a regularization term. When $\gamma \rightarrow 0$, we have the traditional LDA procedure \cite{fisher1936use}, and when $\gamma \rightarrow \infty$,  we recover PCA. 
As opposed to PCA which optimizes the least squares reconstruction, we now have an orthogonal decomposition into multiple components $\widehat{\phi_k}$ that optimizes for the most discriminating directions between the classes.
Figure \ref{fig:tbm_fig_2} shows an example of TBM using PLDA applied to discover discriminating information between benign and malignant cancer nuclei. The linear classifier $\phi_1$ (the first direction retrieved by the PLDA procedure) can be reconstructed back in image space as before by inverting the line $\widehat{f_{\alpha,1}}=\bar{m}(x) + \alpha \widehat{\phi_1}$.
The visualization is provided in the right portion of the figure, demonstrating that malignant nuclei tend to have their chromatin more evenly spread throughout the nuclei than their benign counterparts. Whereas the benign nuclei tend to have their chromatin more concentrated near the periphery of the nucleus. 

The TBM procedure is not limited to PCA, LDA, or PLDA. In fact, it can be used more generally. Instead of using binary regression (classification) as in the example above, we can also perform regression onto any continuous labels $y_1,y_2,\dots$ that may be available in the original dataset. Canonical component analysis (CCA), for example, can be used to find canonical components $\widehat{w_k},v_k$ in transport transform space and label space, respectively. 
The canonical components $\widehat{w_k},v_k$ 
maximize the correlation coefficients of the projections $\text{Corr}(\widehat{w_k}^T \widehat{f}, v_k^T y)$. As such they can be used to recover meaningful relationships between transport modes and functional data. Once the relationships are obtained, they can also be reconstructed and visualized back in signal/image space. Examples are omitted here for brevity, but they can be found in \cite{basu2014detecting,kundu2018discovery,kundu2019assessing,kundu2021investigating} where subcellular mass concentrations have been correlated with drug dosage, brain white and gray matter measurements correlated with ageing, cardiorespiratory fitness, reaction times and others.

\begin{figure}[h!]
    \centering
    \includegraphics[width=0.95\linewidth]{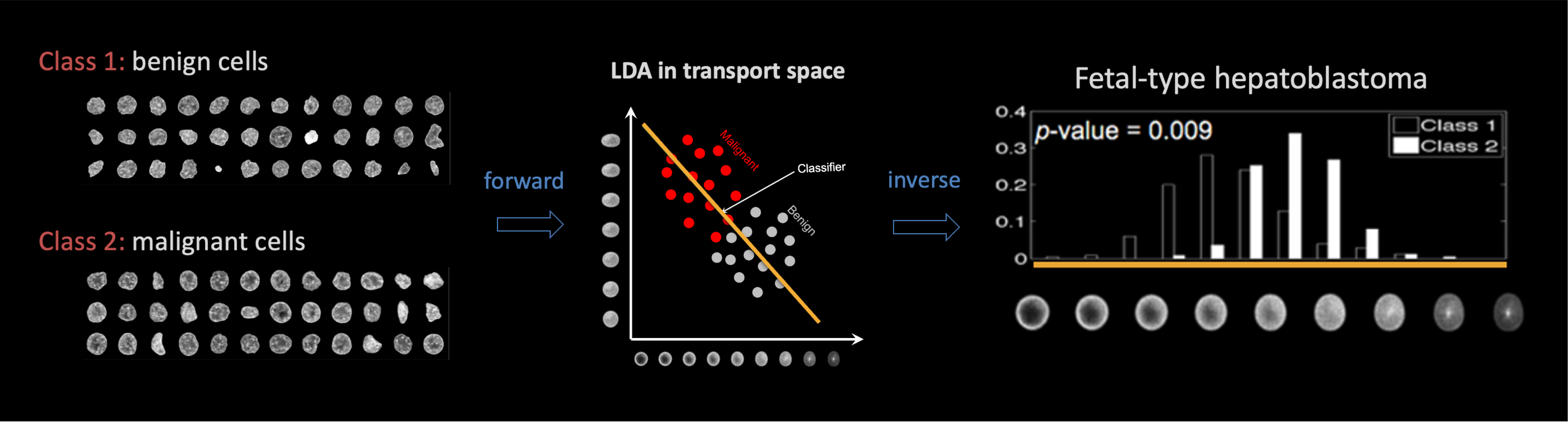}
    \caption{TBM is applied to model discriminating information between nuclei from benign and malignant tissues. Nuclear images are first transformed into transport space, where penalized linear discriminant analysis is applied. The resulting classifier (orange line) is inverted and visualized in image space, along with the projection of the dataset. In this case, TBM suggests the hypothesis that in malignant nuclei, nuclear chromatin is often more dispersed throughout the nucleus, in contrast to benign nuclei, which have more chromatin positioned at the periphery of the nucleus. Figure adapted from \cite{basu2014detecting}. }
    \label{fig:tbm_fig_2}
\end{figure}


\section{Summary}
This chapter reunites the most important properties of the Linear Optimal Transport (LOT) transform together with its variants based on Euclidean one-dimensional OT. We decided to write it using measure theory since it greatly simplifies the notation and is better suited to understanding the spaces where these transforms are defined. We took the viewpoint of starting with a bijective map into a Euclidean space (based on OT) and then copying the structure from the Euclidean space. Although in some cases this bijection has a very clear interpretation from OT, in some others it is a combination of techniques that aim to balance the most useful aspects of Optimal Transport and Hilbert Spaces methods. The simplicity of this approach is sometimes overlooked in the literature in favor of linearizing in the strict sense of finding a linear geometric structure of different variants of OT. The exposition is not exhaustive since other transport-based transforms (or frameworks) exist. For example, Linearized Optimal Transport on manifolds \cite{sarrazin2023linearized} which includes the Linear Circular Optimal Transport (LCOT) embedding \cite{martin2023lcot} (i.e., LOT for one-dimensional circular domains), Linearized Hellinger-Kantorovich (LHK) approach \cite{cai2022linearized}, Linear Optimal Partial Transport (LOPT) embedding \cite{bai2023linear}, linear Gromov-Wasserstein distance \cite{beier2021linear}.

\section*{Acknowledgments}
This work was funded in part by ONR award N000142212505 and NIH award GM130825. Authors wish to offer Professor Akram Aldroubi their sincere thanks for his teaching and support throughout many years.

\section*{Appendix}

\addcontentsline{toc}{section}{Appendix}


\noindent \textbf{LOT-geodesics.}
Consider the curve of measures $\rho_t:=\left((1-t)\widehat{\nu_1}+t\widehat{\nu_2}\right)_{\#}\mu_r$ for $t\in[0,1]$.
 First, when $t=0$, $\rho_0=(\widehat{\nu_1})_\#\mu_r=\nu_1$, and when $t=1$, $\rho_1=(\widehat{\nu_2})_\#\mu_r=\nu_2$.
 Secondly, for each fixed $t$, by Brenier's Theorem the map $(1-t)\widehat{\nu_1}+t\widehat{\nu_2}$ is the optimal transport map between $\mu_r$ and $\rho_t$, since it pushes $\mu_r$ to $\rho_t$ by definition of $\rho_t$, and it is convex because is a non-negative linear combination of convex functions.
 Finally, $\rho_t$ is a constant speed geodesic:
 \begin{align*}\label{eq: geodesic lot}
    d_{LOT}(\rho_t,\rho_s)&=\|\widehat{\rho_t}-\widehat{\rho_s}\|_{L^2(\mu_r)}\\
    &=\|(1-t)\widehat{\nu_1}+t\widehat{\nu_2}-(1-s)\widehat{\nu_1}-s\widehat{\nu_2}\|_{L^2(\mu_r)}\\
    &=\|(s-t)(\widehat{\nu_1}-\widehat{\nu_2})\|_{L^2(\mu_r)}=|s-t|d_{LOT}(\nu_1,\nu_2).
    \end{align*}

\noindent\textbf{Characterization property of CDT.}

\begin{proof}[Proof of Theorem \ref{thm: characterization property 1d}]
    Since $\widehat{\nu}$ is non-decreasing, by Theorem \ref{thm: Brenier} and \eqref{eq: gen inverse comp with cdf}, it is the transport map from the reference $\mu_r:=(\widehat{\nu}^\dagger)_\#\nu$ to $\nu$ (since $(\widehat{\nu}\circ\widehat{\nu}^\dagger)_\#\nu=\nu$).
    
    Now, given any $\sigma\in\mathcal{P}(I)$ there exists a unique non-decreasing $g$  such that $g_\#\nu = \sigma$ (such function $g$ is the unique optimal Monge map). 
    Thus, by property \eqref{eq: property ivan inverse},  $\widehat{\sigma}=\widehat{g_\#\nu}=g\circ \widehat{\nu}$. 
    
    Since $g\circ \widehat{\nu}$ is non-decreasing and satisfies $(g\circ \widehat{\nu})_\# \mu_r = \sigma$, using the uniqueness in Theorem \ref{thm: Brenier}, we have that $\widehat{\sigma}= g\circ \widehat{\mu}$  is the optimal transport map from $\mu_r$ to $\sigma$. This proves that the transform of any $\sigma$ is exactly the CDT with respect to the same reference $\mu_r$.
\end{proof}

\noindent\textbf{Geodesic property of SCDT.}
The curve $\rho_t$ is a curve of Radon measures, that is, $\rho_t\in \mathcal{M}_+(I)$ with 
$$|\rho_t|=(1-t)|\nu_1|+t|\nu_2|$$
since $\left((1-t)\widehat{\nu_1^N}+t\widehat{\nu_2^N}\right)_\#\mathcal{L}_{[0,1]}$ is a probability measure. 
Thus, given $s,t\in[0,1]$    
    \begin{equation*}        d_{SCDT}^2(\rho_t,\rho_s)=d_{CDT}^2(\rho_t^N,\rho_s^N)+\left(|\rho_t|-|\rho_s|\right)^2.
    \end{equation*}
On the one hand,
\begin{equation*}
    \left||\rho_t|-|\rho_s|\right|=|t-s|\left||\nu_1|-|\nu_2|\right|.
\end{equation*}
On the other hand, notice that
$$\rho_t^N=\left((1-t)\widehat{\nu_1^N}+t\widehat{\nu_2^N}\right)_\#\mathcal{L}_{[0,1]}$$
is the geodesic curve between $\nu_1^N$ and $\nu_2^N$ with respect to $d_{CDT}$ and $W_2$. Thus,
\begin{equation*}
    d_{CDT}(\rho_t^N,\rho_s^N)=|t-s|d_{CDT}(\nu_1^N,\nu_2^N).
\end{equation*}
Combining all together we obtain:
\begin{align*}
d_{SCDT}^2(\rho_t,\rho_s)&=(t-s)^2d_{CDT}^2(\nu_1^N,\nu_2^N)+(t-s)^2\left(|\nu_1|-|\nu_2|\right)^2\\
&=(t-s)^2d_{SCDT}^2(\nu_1,\nu_2).
\end{align*}

\noindent\textbf{Inverse formula and symmetry property of RCDT.}

\begin{remark}\label{rem: desint}
We notice that if $\mu$ is a probability measure on $\mathbb{R}^2$, then
$\nu:=\mathcal{R}(\mu)$ is a probability measure defined on $\mathbb{R}\times [0,\pi]$. 
By the disintegration theorem in classic measure theory, there exists a $\nu-$a.s. unique set of measures $(\nu_\theta)_{\theta\in [0,\pi]}\subset \mathcal{M}(\mathbb{R})$ such that for any $\psi \in C_0(\mathbb{R}\times[0,\pi])$, we have 
\begin{align}
    \int_{\mathbb{R}\times[0,\pi]}\psi(t,\theta) \, d\nu(t,\theta)=\int_{[0,\pi]}\int_{\mathbb{R}}\psi(t,\theta) \, d\nu_\theta(t) \, d\theta \label{eq: nu_theta}.
\end{align}
Given a fixed $\theta\in[0,\pi]$, $\mathcal{R}(\mu)_\theta$ describes the projected measure of $\mu$ into the 1D space indexed by $\theta$ \cite[Proposition 6]{bonneel2015sliced}, that is: 
$$\mathcal{R}(\mu)_\theta=\nu_\theta=\theta_\#^* \mu\in\mathcal{P}(\mathbb{R}).$$
By abuse of notation, we will use $$\nu_\theta=\nu(\cdot,\theta) \quad \text{ and so } \quad \mathcal{R}(\mu)_\theta=\mathcal{R}(\mu)(\cdot,\theta)=\theta_\#^* \mu.$$
i.e., $(\mathcal{R}(\mu)(\cdot,\theta))_{\theta\in[0,\pi]}$ also denotes the `disintegration' of $\mathcal{R}(\mu)$.
\end{remark}

\begin{remark}
    By using disintegration of measures, the inverse formula of RCDT can be written as 
    \begin{align*}
        T\mapsto \mathcal{R}^{-1}\left[\left(T(\cdot,\theta)_\#\mathcal{L}_{|_{[0,1]}}\right)_{\theta\in[0,\pi]}\right]=\mathcal{R}^{-1}\left[\left(CDT^{-1}(T(\cdot,\theta))\right)_{\theta\in[0,\pi]}\right].
    \end{align*}
\end{remark}

Now, we can decide the formula of the inverse RCDT:
    The inverse formula of the RCDT follows from the identity:
    \begin{equation}
        T(\cdot,\theta)_\#[\sigma(\cdot,\theta)]=[L^T_\#\sigma](\cdot,\theta)
    \end{equation}
    that holds for every measure $\sigma\in \mathcal{P}(\Omega)$, $\Omega\subset\mathbb R^2$. 
Notice that, in particular, for $\sigma=\mathcal{L}_{|_{[0,1]\times [0,\pi]}}$, we have $\mathcal{L}_{|_{[0,1]\times [0,\pi]}}(\cdot,\theta)=\mathcal{L}_{|_{[0,1]}}$.
    

Then, the symmetry property of the RCDT can be deduced as follows:
By using Remark \ref{rem: desint}, we obtain
\begin{align*}
[\mathcal{R}\nu^g](\cdot,\theta)
&=[\mathcal{R}\mathcal{R}^{-1}(L^g_\#\mathcal{R}\nu)](\cdot,\theta)=
[L^g_\#\mathcal{R}\nu](\cdot,\theta)\\
&= {(g_\theta)}_\# \left([\mathcal{R}v](\cdot,\theta)\right) = {(g_\theta)}_\# \left([\mathcal{R}v]_\theta\right) .
\end{align*}
By taking the CDT of such dimensional measure, and using the symmetry property of the CDT, we get
\begin{equation*}
    \widehat{{(g_\theta)}_\# \left([\mathcal{R}v]_\theta\right) }= g_\theta \circ \widehat{[\mathcal{R}\nu]_\theta} .
\end{equation*}
Thus,
\begin{align*}
    \widehat{\nu^g}^{SCDT}(\theta,\xi)&=
    \widehat{[\mathcal{R}\nu^g](\cdot,\theta)}(\xi) = g_\theta(\widehat{[\mathcal{R}v]_\theta}(\xi))\\
    &=g(\widehat{[\mathcal{R}v](\cdot,\theta)}(\xi), \theta)=g(\widehat{\nu}^{SCDT}(\xi,\theta), \theta).
\end{align*}

\noindent\textbf{RCDT-distance is independent from the choice of a reference.}
Here we will discuss Remark \ref{remark: rcdt not depend ref}.
First, notice that the RCDT is characterized by satisfying the property
\begin{equation}\label{eq: radon CDT integral relation}
    \int_{-\infty}^{\widehat{f_\nu}^{RCDT}(\xi,\theta)}\mathcal{R}f_\nu(t,\theta) dt=\xi=\int_0^\xi 1 \, dt \qquad \forall \theta\in\mathbb [0,\pi],\,  \xi\in[0,1]
\end{equation}
In the literature (see, for e.g., \cite{kolouri2015radon,gong2023radon}), the left-most RHS is taken in a more general way:
Given a reference $r$ (a function $r:B(0,R')\to \mathbb R_{\geq 0}$ or a measure $r$ supported on a ball $B(0,R')$), we consider the Radon Cumulative Distribution Transform with respect to $r$, denoted by $RCDT_r$, as the transformation that satisfies the property
\begin{equation*}
    \int_{-\infty}^{\widehat{f_\nu}^{RCDT_r}(\xi,\theta)}\mathcal{R}f_\nu(t,\theta) dt=\int_0^\xi \mathcal{R}r(t,\theta) dt =F_{\mathcal{R}r(\cdot,\theta)}(\xi)
\end{equation*}
where $F_{\mathcal{R}r(\cdot,\theta)}$ is the CDF o the probability density $\mathcal{R}r(\cdot,\theta)$.
So, 
\begin{align*}
    \widehat{f_\nu}^{RCDT_r}(\xi,\theta)
    &=\widehat{[\mathcal{R}f_\nu](\cdot,\theta)}^{CDT_{[\mathcal{R}r](\cdot,\theta)}}(\xi), \qquad \xi \in [-R',R'], \theta\in[0,\pi]
\end{align*}
The inverse formula in this case is 
\begin{align*}
    T\mapsto \mathcal{R}^{-1}(L^T_\#(\mathcal{R}r)) \qquad \text{ where } L^T(t,\theta)=(T(t,\theta),\theta).
\end{align*}
The new RCDT-distance with respect to the reference $r$ is given by
\begin{align*}
     d_{RCDT_r}^2&(\nu_1,\nu_2) =\|\widehat{\nu_1}^{RCDT_r}-\widehat{\nu_2}^{RCDT_r}\|^2_{L^2(d\mathcal{R}r)}\\
     &= \frac{1}{\pi}\int_0^\pi \int_{-R'}^{R'} |\widehat{\nu_1}^{RCDT_r}(\xi,\theta)-\widehat{\nu_2}^{RCDT_r}(\xi,\theta)|^2 [\mathcal{R}r](\xi,\theta) \, d\xi d\theta\\
     &=\frac{1}{\pi}\int_0^\pi d^2_{CDT_{\mathcal{R}r(\cdot,\theta)}}(\mathcal{R}\nu_1(\cdot,\theta), \mathcal{R}\nu_1(\cdot,\theta)) \ d\theta\\
     &=\frac{1}{\pi}\int_0^\pi d^2_{CDT}(\mathcal{R}\nu_1(\cdot,\theta), \mathcal{R}\nu_1(\cdot,\theta)) \ d\theta=SW_2^2(\nu_1,\nu_2)=d_{RCDT}^2(\nu_1,\nu_2)
\end{align*}
thus, the RCDT-distance does not depend on the choice of the reference. 
\\

\noindent\textbf{Inverse formula of RSCDT.}
By using disintegration of measures, the inverse formula of the RCDT can be written as follows:
\begin{align*}
    ((T,a),(U,b))\mapsto& \mathcal{R}^{-1}\left[\left(SCDT^{-1}\left(T(\cdot,\theta),a(\theta),U(\cdot,\theta),b(\theta)\right)\right)_{\theta\in[0,\pi]}\right]\\
    &=\mathcal{R}^{-1}\left[\left(a(\theta)\left(T(\cdot,\theta)_\#\mathcal{L}_{|_{[0,1]}}\right)-b(\theta)\left(U(\cdot,\theta)_\#\mathcal{L}_{|_{[0,1]}}\right)\right)_{\theta\in[0,\pi]}\right].
\end{align*}

\bibliographystyle{spmpsci}
\bibliography{references}

\end{document}